\documentclass[reqno]{amsart}
\usepackage{amsfonts,amssymb,amsmath,amsthm}
\usepackage{url}
\usepackage{enumerate}
\usepackage{pifont}
\usepackage[all]{xy}
\usepackage{hyperref}
\usepackage{cleveref}
\usepackage{caption}

\newtheorem{theorem}{Theorem}[section]
\newtheorem{lemma}[theorem]{Lemma}
\newtheorem{proposition}[theorem]{Proposition}

\theoremstyle{definition}
\newtheorem{definition}[theorem]{Definition}
\newtheorem{remark}[theorem]{Remark}
\newtheorem{notation}{Notation}
\numberwithin{equation}{section}

\newcommand{\cmark}{\ding{51}}%
\newcommand{\xmark}{\ding{55}}%

\global\long\def\C{\mathbb{C}}
\global\long\def\N{\mathbb{N}}
\global\long\def\R{\mathbb{R}}
\global\long\def\H{\mathbb{H}}
\global\long\def\Z{\mathbb{Z}}
\global\long\def\Q{\mathbb{Q}}

\global\long\def\mf#1{\mathfrak{#1}}
\global\long\def\mc#1{\mathcal{#1}}

\global\long\def\im{\mathrm{im}}
\global\long\def\id{\mathrm{id}}
\global\long\def\.{,\dots ,}

\global\long\def\Aut{\operatorname{Aut}}
\global\long\def\Lie{\operatorname{Lie}}
\global\long\def\so{\mathfrak{so}}

\global\long\def\ad{\mathrm{ad}}

\global\long\def\su{\mathfrak{su}}
\global\long\def\so{\mathfrak{so}}
\global\long\def\sp{\mathfrak{sp}}
\global\long\def\sl{\mathfrak{sl}}

\global\long\def\b{\mathfrak{b}}

\global\long\def\g{\mathfrak{g}}
\global\long\def\h{\mathfrak{h}}

\global\long\def\l{\mathfrak{l}}
\global\long\def\m{\mathfrak{m}}

\global\long\def\s{\mathfrak{s}}

\global\long\def\Der{\operatorname{Der}}
\global\long\def\rank{\operatorname{rank}}
\global\long\def\End{\operatorname{End}}

\author{Reinier Storm}
\address{KU Leuven, Department of Mathematics, Celestijnenlaan 200B -- Box 2400, BE-3001 Leuven, Belgium} 
\email{reinier.storm@kuleuven.be}
\thanks{The author is supported by project 3E160361 of the KU Leuven Research Fund.}

\begin{document}

\title{The classification of 7- and 8-dimensional naturally reductive spaces}

\begin{abstract}
A new method for classifying naturally reductive spaces is presented. This method relies on the structure theory of naturally reductive spaces developed in \cite{Storm2018a} and the new construction of naturally reductive spaces in \cite{Storm2018}. We obtain the classification of all naturally reductive spaces in dimension 7 and 8.
\end{abstract}
\maketitle
\section{Introduction}
A locally naturally reductive space is a Riemannian manifold together with a metric connection which has parallel skew-torsion and parallel curvature. A locally symmetric space can thus be seen as a naturally reductive space with zero torsion. In the seminal paper \cite{CartanE1926} Cartan classified all symmetric space. 

The story for naturally reductive spaces is quite different of course. A nice concise classification as for the symmetric spaces is not available. A good place to start is to classify them in small dimensions. This has been done in the dimensions $3,~4,~5$ in \cite{TricerriVanhecke1983,KowalskiVanhecke1983,KowalskiVanhecke1985} and more recently in dimension $6$ in \cite{AgricolaFerreiraFriedrich2015}. These classifications essentially rely on being able to parametrize the possible torsion and curvature tensors of the naturally reductive connections and then solving the first Bianchi identity.
Such a parametrization breaks down in higher dimensions. The recent developments in \cite{Storm2018} and \cite{Storm2018a} tell us however that naturally reductive spaces are still very rigid. 
This gives us the ability to present here a completely new way to classify naturally reductive spaces.

\subsection{Results} 
The new approach presented here can be applied in any dimension, but becomes increasingly more elaborate as the dimension increases. Therefore, it becomes important to find ways to limit the possible cases. This will be carried out explicitly for naturally reductive spaces of dimension $7$ and $8$. 
An important point is the division of naturally reductive spaces into two types as in \cite{Storm2018a}:
\begin{enumerate}[\hspace{20pt} Type I:]
	\item The transvection algebra is semisimple.
	\item The transvection algebra is not semisimple.
\end{enumerate}
Another simplification of the classification comes from the partial duality of naturally reductive spaces defined in \cite{Storm2018a}. This makes the classification much more transparent.
For the spaces of type I we will only list the compact ones and in case a non-compact partial dual space exists we will mention this. For the spaces of type II we will only list the ones for which the semisimple factor of the canonical base space is compact and we will mention if a partial dual spaces exist. 
The classification result of all $7$- and $8$-dimensional naturally reductive spaces is summarized in \Cref{thm:classification type I} for type I and in \Cref{thm:classification type II} for type II.

\section{preliminaries}
The essential structure of a locally homogeneous space is encoded in the infinitesimal model. We now briefly discuss this below.

\begin{theorem}[Ambrose-Singer, \cite{AmbroseSinger1958}]
	A complete simply connected Riemannian manifold $(M,g)$ is a homogeneous Riemannian manifold if and only if there exists a metric connection $\nabla$ with torsion $T$ and curvature $R$ such that 
	\begin{equation}
	\nabla T = 0 \quad \mbox{and}\quad \nabla R = 0. \label{eq:AS connection}
	\end{equation}
\end{theorem}

A metric connection satisfying \eqref{eq:AS connection} is called an \emph{Ambrose-Singer connection}. The torsion $T$ and curvature $R$ of an Ambrose-Singer connection evaluated at a point $p\in M$ are linear maps
\begin{equation}
T_p: \Lambda^2 T_p M \to T_p M,\quad  R_p: \Lambda^2 T_p M \to \so(T_p M),\label{eq:(T,R)}
\end{equation}
which satisfy for all $x,y,z\in T_pM$
\begin{align}
&R_p(x,y)\cdot T_p = R_p(x,y)\cdot R_p=0 \label{eq:par T and R}, \\
&\mf{S}^{x,y,z} R_p(x,y)z - T_p(T_p(x,y),z) = 0  \label{eq:B1},\\ 
&\mf{S}^{x,y,z} R_p(T_p(x,y),z) =0  \label{eq:B2},
\end{align}
where $\mf{S}^{x,y,z}$ denotes the cyclic sum over $x,y$ and $z$ and $\cdot$ denotes the natural action of $\so(T_p M)$ on tensors. The first equation encodes that $T$ and $R$ are parallel objects for $\nabla$ and under this condition the first and second Bianchi identity become equations \eqref{eq:B1} and \eqref{eq:B2}, respectively. A pair of tensors $(T,R)$, as in \eqref{eq:(T,R)}, on a vector space $\m$ with a metric $g$ satisfying \eqref{eq:par T and R}, \eqref{eq:B1} and \eqref{eq:B2} is called an \emph{infinitesimal model} on $(\m,g)$. From the infinitesimal model $(T,R)$ of a homogeneous space one can construct a homogeneous space with infinitesimal model $(T,R)$. This construction is known as the \emph{Nomizu construction}, see \cite{Nomizu1954}. This construction goes as follows.
Let 
\[
\h := \{h\in \so(\m) : h\cdot T=0,~ h\cdot R=0\}.
\]
and set
\begin{equation}\label{eq:Nomizu Lie algebra}
\g := \h \oplus \m.
\end{equation}
On $\g$ the following Lie bracket is defined for all $h,k\in \h$ and $x,y\in \m$:
\begin{equation}
[h + x,k + y] := [h,k]_{\so(\m)} - R(x, y) + h(y) - k(x) - T(x,y),\label{eq:Nomizu Lie bracket}
\end{equation}
where $[-,-]_{\so(\m)}$ denotes the Lie bracket in $\so(\m)$. The bracket from \eqref{eq:Nomizu Lie bracket} satisfies the Jacobi identity if and only if $R$ and $T$ satisfy the equations \eqref{eq:par T and R}, \eqref{eq:B1} and \eqref{eq:B2}. We will call $\g$ the \emph{symmetry algebra} of the infinitesimal model $(T,R)$. Let $G$ be the simply connected Lie group with Lie algebra $\g$ and let $H$ be the connected subgroup with Lie algebra $\h$. The infinitesimal model is \emph{regular} if $H$ is a closed subgroup of $G$. If this is the case, then clearly the canonical connection on $G/H$ has the infinitesimal model $(T,R)$ we started with.
In \cite[Thm.~5.2]{Tricerri1992} it is proved that an infinitesimal model coming from an invariant connection on a globally homogeneous Riemannian manifold, as in \eqref{eq:(T,R)}, is regular.

\begin{definition}\label{def:nat red decomp}
	Let $(\g = \h \oplus \m,g)$ be a Lie algebra together with a subalgebra $\h\subset \g$, a complement $\m$ of $\h$ and a metric $g$ on $\m$. Suppose $\ad(\h)\m\subset \m$ and for all $x,y,z\in \m$ that
	\[
	g([x,y]_\m,z) = -g(y,[x,z]_\m).
	\]
	Then we call $(\g=\h\oplus \m,g)$ a \emph{naturally reductive decomposition} with $\h$ the \emph{isotropy algebra}. We will mostly refer to just $\g=\h\oplus \m$ as a naturally reductive decomposition and let the metric be implicit. The infinitesimal model of the naturally reductive decomposition is defined by 
	\begin{align}
	T(x,y)&:= -[x,y]_\m,                &\quad \forall x,y\in \m,\label{eq:torsion from lie bracket} \\
	R(x,y)&:=-\ad([x,y]_\h) \in \so(\m), &\quad \forall x,y\in \m,\label{eq:curvature from lie bracket} 
	\end{align}
	where $[x,y]_\m$ and $[x,y]_\h$ are the $\m$- and $\h$-component of $[x,y]$, respectively. We call the decomposition an \emph{effective} naturally reductive decomposition if the restricted adjoint map $\ad:\h\to \so(\m)$ is injective. We will say that $\g$ is the \emph{transvection algebra} of the naturally reductive decomposition $\g=\h\oplus \m$ if the decomposition is effective and $\im(R) = \ad(\h) \subset \so(\m)$. Note that \eqref{eq:par T and R} implies that $\im(R)\subset \so(\m)$ is a subalgebra and that the transvection algebra is always a Lie subalgebra of the symmetry algebra. When we simply refer to $\g=\h\oplus \m$ as a \textit{naturally reductive transvection algebra}, then we mean that this is a naturally reductive decomposition for which $\g$ is also the transvection algebra.
\end{definition}

The following theorem is a classical result by Kostant (see also \cite{D'AtriZiller1979}).
\begin{theorem}[Kostant,\cite{Kostant1956}] \label{thm:kostant}
	Let $(\g=\h\oplus \m,g)$ be an effective naturally reductive decomposition. Then $\mf{k}:= [\m,\m]_{\h}\oplus \m$ is an ideal in $\g$ and there exists a unique $\ad(\mf{k})$-invariant non-degenerate symmetric bilinear form $\overline{g}$ on $\mf{k}$ such that $\overline{g}|_{\m\times\m}=g$ and $[\m,\m]_\h \perp \m$. Conversely, any $\ad(\g)$-invariant non-degenerate symmetric bilinear form on $\g=\h\oplus \m$ with $\m=\h^\perp$ and $\overline{g}|_{\m\times\m}$ positive definite gives a naturally reductive decomposition.
\end{theorem}

This motivates the following definitions.

\begin{definition}\label{def:nr pair}
	A Lie algebra $\g$ together with a subalgebra $\h\subset \g$ define a \emph{naturally reductive pair} $(\g,\h)$ if there exists an $\ad(\g)$-invariant non-degenerate symmetric bilinear form $\overline{g}$  for which $\overline{g}|_{\m\times \m}$ is positive definite, where $\m = \h^\perp$ and such that $\g$ is the transvection algebra of the corresponding naturally reductive decomposition.
\end{definition}

\begin{definition}\label{def:dual pair}
	A naturally reductive pair $(\g^*,\h^*)$ is a \emph{partial dual} of a naturally reductive pair $(\g,\h)$ when $\g^*$ is a real form of $\g \otimes \C$ different from $\g$ and the complexified Lie algebra pairs are isomorphic: $(\g\otimes  \C,\h\otimes  \C)\cong (\g^*\otimes \C,\h^*\otimes \C)$.
\end{definition}

\begin{remark}\label{rem:recall dual}
	In \cite{Storm2018a} it is shown that every naturally reductive pair of type I admits exactly one compact partial dual pair of type I. Moreover, for spaces of type II there exists exactly one partial dual pair for which the semisimple part of the canonical base space is compact. 
	Moreover, a partial dual of naturally reductive pair $(\g,\h)$ of type I exists if and only if $(\g_i,\mbox{proj}_{\g_i}(\h))$ is a symmetric pair for some $i$, where $\g = \g_1\oplus \g_2 \oplus \dots \oplus \g_n$ is the decomposition of $\g$ into simple ideals.	
	
	We will use this partial duality to deflate our classification list and make it more comprehensive.
\end{remark}

In \cite{Storm2018} a new construction of naturally reductive spaces is defined. This starts from a naturally reductive transvection algebra and a certain subalgebra $\mf k$ of derivations of the transvection algebra together with an $\ad(\mf k)$-invariant metric $B$ on it and constructs a new naturally reductive decomposition. This new decomposition we call the \textit{$(\mf k,B)$-extension} and generally this new space is a homogeneous fiber bundle over the original space.
More explicitly, for $\g=\h\oplus \m$ a non-zero transvection algebra the algebra $\mf k$ of derivations has to be  a subalgebra of
\begin{equation}\label{eq:s(g) definition}
\s(\g) := \{f\in \Der(\g):f(\h)=\{0\},~f(\m)\subset \m,~f|_\m\in \so(\m)\}.
\end{equation}
If $\g=\{0\}$, then we define $\s(\{0\}) = \so(\infty)$.
For every finite dimensional subalgebra $\mf k\subset \s(\g)$ with an $\ad(\mf k)$-invariant metric $B$ on $\mf k$ a new naturally reductive decomposition is obtained which is called to $(\mf k,B)$-extension, see \cite{Storm2018}. 

One can imagine that spaces of type I are relatively easy to classify. The theory developed in \cite{Storm2018a} together with the construction of \cite{Storm2018} allow us to classify all remaining naturally reductive spaces. The main reason for this is the following slightly rephrased result from \cite{Storm2018a}.

\begin{theorem}[{\cite[Thm.~4]{Storm2018a}}]\label{thm:canonical base}
	For every naturally reductive decomposition of type II there exists a unique naturally reductive transvection algebra of the form
	\begin{equation}\label{eq:unique base decomp}
	\g=\h\oplus \m \oplus_{L.a.} \R^n,
	\end{equation}
	with $\h\oplus \m$ a semisimple algebra and $\oplus_{L.a.}$ denotes the direct sum of Lie algebras, such that the original type II decomposition is a $(\mf k,B)$-extension of $\g=\h\oplus \m \oplus_{L.a.} \R^n.$
\end{theorem}

\begin{definition}
	The unique naturally reductive decomposition from \Cref{thm:canonical base} in \eqref{eq:unique base decomp} is called the \emph{canonical base space} of the type II space.
\end{definition}

The following propositions will simplify the classification procedure tremendously, see \cite{Storm2018a} for the proofs.
We require that all of the spaces we list in the classification are irreducible. The first proposition deals with this for all spaces of type II.

\begin{proposition}\label{lem:reducibility criteria}
	Let $\g = \h \oplus \m \oplus_{L.a.} \R^n$ be a naturally reductive transvection algebra with $\h\oplus \m$ semisimple. Furthermore, let $\mf k\subset \s(\g)$ and let $B$ be some $\ad(\mf k)$-invariant inner product on $\mf k$. Consider the following decomposition
	\begin{equation}\label{eq:irreducible factor}
	\g=(\h_1\oplus\m_1)\oplus_{L.a.} \dots \oplus_{L.a.} (\h_p \oplus \m_p) \oplus_{L.a.} \m_{p+1} \oplus_{L.a.} \dots \oplus_{L.a.} \m_{p+q},
	\end{equation} 
	where $\h_i\oplus \m_i$ is an irreducible naturally reductive decomposition with $\h_i\subset \h$ and $\m_i\subset \m$ for $i=1,\dots, p$ and $\m_{p+j}\subset \R^n$ is an irreducible $\mf k$-module for $j=1,\dots, q$. We choose the $\m_1,\dots ,\m_{p+q}$ mutually orthogonal. Suppose that $\g=\h\oplus \m\oplus_{L.a.} \R^n$ is the canonical base space of the $(\mf k,B)$-extension. The $(\mf k,B)$-extension is reducible if and only if there exists a non-trivial partition:
	\[
	\{\m_1,\dots,\m_p,\m_{p+1},\dots,\m_{p+q}\}=W'\cup W'', \quad W'\cap W'' = \emptyset,
	\]
	and an orthogonal decomposition of ideals $\mf k=\mf k'\oplus \mf k''$ with respect to $B$ such that $\mf k'$ acts trivially on all elements of $W''$ and $\mf k''$ acts trivially on all elements of $W'$.
\end{proposition}

We also need to recall some definitions from \cite{Storm2018}.

\begin{definition}\label{def:k_i and b_i and varphi}
	Let $(\g=\h \oplus \m\oplus_{L.a.}\R^n,g)$ be a as in \eqref{eq:unique base decomp} Let $\mf k\subset \s(\g)$ be a Lie subalgebra and let $B$ be an $\ad(\mf k)$-invariant inner product on $\mf k$. Let $\varphi:\mf k\to \so(\m)$ be the natural Lie algebra representation and let $\psi:\mf k \to \so(\mf k\oplus\m)$ be the Lie algebra representation $\psi:=\ad\oplus \varphi$.

	Then we define $\varphi_1:\mf k\to\so(\m)$ and $\varphi_2:\mf k \to \so(\R^n)$ to be the restricted representations of $\varphi:\mf k\to \so(\m)$. Next we put
	\[
	\mf k_1 := \ker(\varphi_2),\quad \mf k_3 := \ker(\varphi_1),\quad \mf k_2 := (\mf k_1\oplus \mf k_3)^\perp \subset \mf k,
	\]
	where the orthogonal complement is taken with respect to $B$.
	Furthermore, recall that $\mf s(\g) \cong \mc{Z}(\h)\oplus \mf p\oplus \so(n)$, where $\mf p := \{m\in \m:[h,m]=0,~\forall h\in \h\}$ and $\mc{Z}(\h)$ denotes the center of $\h$. In this way we identify $\mf k_1\oplus \mf k_2\subset \Aut(\h\oplus \m)$ with inner derivations: $\b_1\oplus \b_2 \subset \mc{Z}(\h)\oplus \mf p\subset \h\oplus \m$.
\end{definition}

All the spaces of type II are constructed as $(\mf k,B)$-extensions. Generically a $(\mf k,B)$-extension results in a space of type II. However, in general it does not. The following proposition guaranties that all of the $(\mf k,B)$-extensions we list are of type II. This is to assure that none of the spaces we list are isomorphic.

\begin{proposition}\label{lem:canonical base iff}
	Let $\mf f$ be the transvection algebra of a $(\mf k,B)$-extension of a naturally reductive transvection algebra of the form $\g=\h\oplus \m \oplus_{L.a.} \R^n$. The canonical base space of $\mf f$ is isomorphic to $\g=\h\oplus \m \oplus_{L.a.} \R^n$ if and only if the following hold
	\begin{enumerate}[$(i)$]
		\item $\pi_\m(\mc Z(\b_1))=\{0\}$,
		\item $\ker(R|_{\ad(\h)+\psi(\mf k)})=\{0\}$,
	\end{enumerate}
	where $\pi_{\m}$ denotes the projection onto $\m$ and $R$ is the curvature tensor associated to $\mf f$.
\end{proposition}

For \Cref{lem:canonical base iff} we need to be able to compute $\ker(R)$ the following lemma simplifies this.

\begin{lemma} \label{lem:ker(R) characterisation}
	Let $\g=\h\oplus \m\oplus_{L.a.} \R^n$ be a naturally reductive transvection algebra. Let $\mf k\subset \s(\g)$ and let $B$ be an $\ad(\mf k)$-invariant inner product on $\mf k$. Let $(T,R)$ be the infinitesimal model of the $(\mf k,B)$-extension. 
	Then
	\[
	\ad(\h^{ss})\oplus\ad(\mf k^{ss})=\ad(\h^{ss}\oplus \mf k^{ss})\subset \im(R) \quad\mbox{and}\quad \ker(R) \subset \ad(\mc Z(\h\oplus\mf k)),
	\]
	where $\g^{ss}$ denotes the semisimple commutator ideal of a Lie algebra $\g$ and $\mc{Z}(\g)$ denotes the center of $\g$. Moreover, if $\mf k_1=\{0\}$, then $\ker(R)=\{0\}$.
\end{lemma}

Finally we need to be able to decide whether two spaces of type II are isomorphic. The following proposition does exactly this.

\begin{proposition}\label{prop:iso type II}
	Let $\g_i=\h_i \oplus \m_i = \h_i\oplus \m_{0,i}\oplus_{L.a.}\R^{n_i}$ be naturally reductive transvection algebras with $\h_i\oplus \m_{0,i}$ semisimple or $0$-dimensional for $i=1,2$. Suppose $\g_i=\h_i\oplus \m_i$ is the canonical base space of some $(\mf k_i,B_i)$-extension for $i=1,2$ and that the $(\mf k_1,B_1)$-extension and $(\mf k_2,B_2)$-extension are isomorphic. Then there is a Lie algebra isomorphism $\tau:\g_1  \to \g_2$.
	Furthermore, $\tau(\h_1) = \h_2$, $\tau|_{\m_1}:\m_1\to \m_2$ is an isometry and $\tau_*:\mf k_1\to \mf k_2$ is an isometry, where $\tau_*:\Der(\g_1)\to\Der(\g_2)$ is the induced map on derivations.
\end{proposition}

The above proposition also covers type I spaces by considering a type I space as a trivial $(\mf k,B)$-extension over itself. Also note that the isomorphism $\tau$ from the above proposition is necessarily an isometry with respect to the unique invariant non-degenerate symmetric bilinear form of \Cref{thm:kostant}.

\section{Classification of type I \label{sec:Classification type I}}

Now we describe how to classify all naturally reductive decomposition of type I in some dimension $k$. First list all semisimple Lie algebras $\g$ of dimension less or equal to $\frac{1}{2}k(k+1)$. For all of these look for subalgebras $\h\subset \g$ together with all $\ad(\g)$-invariant non-degenerate symmetric bilinear forms $\overline{g}$ on $\g$ such that the following hold:
\begin{enumerate}
	\item $\dim(\g/\h)=k$, \label{item dimension}
	\item $\overline{g}|_{\m\times \m}$ is positive definite, where $\m=\h^\perp$, \label{item non-deg}
	\item the torsion $T$ from \eqref{eq:torsion from lie bracket} is irreducible, \label{item torsion irreducible}
	\item $[\m,\m]_\h = \h$.\label{item transvection}
\end{enumerate}
We will refer to these as conditions 1 to 4, as we will use them regularly. Condition \ref{item transvection} implies that $\g$ is the transvection algebra, see \cite[Lem.~8]{Storm2018a}. 
This produces all irreducible naturally reductive transvection algebras $\g=\h\oplus \m$ of type I and thus, after finding all isomorphic ones, we obtain a classification of all naturally reductive transvection algebras of type I in dimension $k$

\begin{remark} 
	The naturally reductive structures on globally homogeneous spaces are the ones which are regular.  
	To obtain all regular structures we only have to investigate when $H$ is closed in $G$, where $G$ is the simply connected Lie group with Lie algebra $\g$ and $H$ is the connected subgroup with Lie subalgebra $\h$, see \cite{Kowalski1990, Tricerri1992}.
	We mention when the naturally reductive structure at hand is regular for all the cases we discuss.
\end{remark}

The above approach is a very crude and already in dimension 7 and 8 this becomes quite some work. We can make our method more efficient by first looking for an upper bound for the dimension of $\h$. We used above that $\h$ is always a subalgebra of $\so(k)$ and thus $\dim(\h)\leq \frac{1}{2}k(k-1)$. However, since $\h$ is the stabilizer of an irreducible $3$-form $T\in \Lambda^3\m$ we can improve this estimate. In \Cref{tab:stabilizer 3-form} we list stabilizers of irreducible $3$-forms in dimension $3$ to $8$ which are of the largest dimension possible.
\begin{table}[h]
	\[
	\begin{array}{|c | c | c | c | c | c | c|}
	\hline
	k & 3 & 4 & 5 & 6 & 7 & 8 \\
	\hline
	\h & \so(3) & \textrm{n/a} & \mf{u}(2) & \su(3) & \g_2 & \su(3)\\
	\hline
	D_k := \dim(\h) & 3 & \textrm{n/a} & 4 & 8 & 14 & 8\\
	\hline
	\end{array}
	\]
	\captionof{table}{Stabilizers of some irreducible 3-form of the largest dimension possible.}\label{tab:stabilizer 3-form}
\end{table}

\noindent Additionally we can also look for the stabilizer with the second largest dimension $d_k$. In dimension $7$ we find this is $\mf u(3)$, which has dimension $9$. In dimension $8$ we find it has at most dimension $5$. Now we can apply the above approach, but only listing the semisimple Lie algebras up to dimension $k + D_k$ and we also don't have to list semisimple Lie algebra $\g$ with $k+d_k\leq \dim(\g) \leq k+D_k$. This is already a big improvement compared to the initial approach.

The next step is to find all subalgebras of these semisimple Lie algebras, such that the conditions \ref{item dimension}-\ref{item transvection} are satisfied. We do this for every semisimple Lie algebra by listing all reductive algebras $\h$ which satisfy $\dim(\h)=\dim(\g)-k$ and $\mbox{rank}(\h)\leq \min\{ \mbox{rank}(\g),\mbox{rank}(\so(k))\}$. Once we have the list of all such pairs $(\g,\h)$ we have to find all possible injective Lie algebra homomorphisms $\h\to \g$ up to conjugation by an automorphism of $\g$, such that the conditions \ref{item torsion irreducible} and \ref{item transvection} are satisfied for some non-degenerate symmetric bilinear form $\overline{g}$.
For condition \ref{item torsion irreducible} it is often easier to check the condition in the following lemma, see \cite[Lem.~5]{Storm2018a} for a proof.

\begin{lemma}\label{lem:reducible iff ideals}
	Let $\g=\h\oplus \m$ be a naturally reductive transvection algebra. Let $\overline{g}$ be the unique $\ad(\g)$-invariant non-degenerate symmetric bilinear form from Kostant's theorem, see {\rm \Cref{thm:kostant}}. The reductive decomposition $\g=\h\oplus \m$ is reducible if and only if there exist two non-trivial orthogonal ideals $\g_1\subset \g$ and $\g_2\subset \g$ with respect to $\overline{g}$ such that $\g = \g_1\oplus \g_2$ and $\h=\h_1\oplus \h_2$ with $\h_i\subset \g_i$ for $i=1,2$.
\end{lemma}

The following lemma is useful to list all conjugacy classes of subalgebras of $\so(n)$ and $\su(n)$ in small dimensions.

\begin{lemma}\label{lem:equivalent reps}
	Let $\g=\so(n)$ or $\g=\su(n)$. Let $\pi:\g\to \End(K^n)$ be the vector representation, with $K=\R^n,~\C^n$. Let $f_i:\h\to \g$ be an injective Lie algebra homomorphism for $i=1,2$. 
	We denote the image of $f_i$ by $\h_i:=f_i(\h)$. If the representations $\pi\circ f_1$ and $\pi\circ f_2$ are equivalent, then the subalgebras $\h_1$ and $\h_2$ are conjugate by an automorphism of $\g$.
\end{lemma}
Note that the above lemma implies the naturally reductive pairs defined by $(\g,\h_1)$ and $(\g,\h_2)$ are isomorphic if and only if the representations representations of $\h_1$ and $\h_2$ are equivalent. The last step is to find all $\ad(\g)$-invariant non-degenerate symmetric bilinear forms on $\g$ such that condition \ref{item non-deg} is satisfied.

Let us briefly illustrate how one can obtain \Cref{tab:stabilizer 3-form} by explaining it in dimension 8. The largest dimensional stabilizer will be a proper subalgebra of $\so(8)$ of dimension bigger than or equal to $8$, since the adjoint representation of $\su(3)$ stabilizes the irreducible 3-form defined by $T(x,y,z):=B_{\su(3)}([x,y],z)$. Note that any stabilizer is a reductive Lie algebra and its commutator ideal is equal to one of the following semisimple Lie subalgebras of $\so(8)$:
\[
\su(2),~\su(2)^2,~\su(3),~\su(2)^3,~ \sp(2),~\so(4)\oplus \so(4),~\sp(2)\oplus \sp(1),~\g_2,~\su(4),~ \so(7).
\]
The only Lie algebras $\h$ with semisimple part $\su(2)$ and $\mbox{rank}(\h)\leq \mbox{rank}(\so(8)) = 4$ are $\h=\su(2)\oplus \R^i$ for $i=1,2,3$. These are of dimension less than or equal to 6. Hence to find the stabilizer with the largest dimension, we can forget about these cases. For the other Lie algebras we list the complexifications of all 8-dimensional real representations and check if there exists an irreducible invariant 3-form. The next step is to check if the representation can be extended to a larger Lie algebra and see if the 3-form is still stabilized by this larger Lie algebra. 

The following lemma will exclude many Lie subalgebras $\h\subset \so(k)$ from having an invariant irreducible 3-form.
\begin{lemma}\label{lem:so(l)<h then no irred 3-form}
	Suppose that $\so(l)\subset \h \subset \so(k)$, where the inclusion $\so(l)\subset \so(k)$ is the standard block embedding and $l\geq 3$. Then there is no $\h$-invariant irreducible 3-form $T\in \Lambda^3 \R^k$.
\end{lemma}
\begin{proof}
	We show that there is no irreducible 3-form invariant under $\so(l)$ and this implies that there is no invariant irreducible 3-form under the $\h$-action. As an $\so(\l)$ module $\R^k$ splits into two orthogonal submodules: $\R^k=\R^l\oplus \R^{k-l}$. This implies that 
	\[
	T\in \Lambda^3 \R^l \oplus \Lambda^2 \R^l \otimes \R^{k-l} \oplus \R^l\otimes \Lambda^2 \R^{k-l} \oplus \Lambda^3 \R^{k-l},
	\]
	and all direct sums are preserved by $\so(l)$. Let $T_2$ denote the component of $T$ in $\Lambda^2 \R^l \otimes \R^{k-l}$. We can identify $T_2$ with an $\so(l)$-equivariant map $T_2:\Lambda^2\R^l\to \R^{k-l}$.
	Since $\so(l)$ acts trivially on $\R^{k-l}$ and has no fixed 2-forms, because $\Lambda^2\R^l\cong \so(l)$ is the adjoint representation. We conclude by Schur's lemma that $T_2=0$. By a similar argument the component of $T$ in $\R^l\otimes \Lambda^2 \R^{k-l}$ vanishes. We conclude $T\in \Lambda^3 \R^l \oplus \Lambda^3\R^{k-l}$ and thus $T$ is reducible.
\end{proof}
Note that $\su(2)^2$ is a subalgebra of the following Lie algebras
\[
\su(2)^3,~ \sp(2),~\so(4)\oplus \so(4),~\sp(2)\oplus \sp(1),~\g_2,~\su(4),~ \so(7).
\]
Therefore, if there is no representation of $\su(2)^2$ that stabilizes an irreducible 3-form, then there is also no representation of any of these Lie algebras which stabilizes and irreducible 3-form. In the following we will denote a highest weight representations of a semisimple Lie algebra $\g$ with highest weight $n_1\lambda_1 +\dots+ n_p\lambda_p$ as $R(n_1,\dots,n_p)$, where $\lambda_1,\dots,\lambda_p$ are the fundamental weights of $\g$ in the Bourbaki labeling. All complexifications of 8-dimension faithful real representations of $\su(2)^2$ are:
\begin{align*}
&R(1,0)\oplus R(0,1),~R(1,0)\oplus R(0,2)\oplus R(0,0),~R(1,1)\oplus 4R(0,0),\\
&R(1,1)\oplus R(0,1),~R(1,1)\oplus R(0,2)\oplus R(0,0),~R(1,1)\oplus R(1,1),\\
&R(4,0)\oplus R(0,2),~R(2,0)\oplus R(0,2)\oplus 2R(0,0).
\end{align*}
For the representations $R(1,0)\oplus R(0,2)\oplus R(0,0)$, $R(1,1)\oplus R(0,2)\oplus R(0,0)$, $R(2,0)\oplus R(0,2)\oplus 2R(0,0)$, $R(1,1)\oplus 4R(0,0)$ and $R(4,0)\oplus R(0,2)$ we can apply \Cref{lem:so(l)<h then no irred 3-form} to see that there is no invariant irreducible 3-form. For the other three representations it follows that there are no irreducible invariant 3-forms by a similar argument as that in \Cref{lem:so(l)<h then no irred 3-form}.

We conclude that the stabilizer of some irreducible 3-form of the largest dimension possible has $\su(3)$ as its commutator ideal. The representation $R(1,1)$ is the complexified adjoint representation of $\su(3)$ and it is of real type. Hence the endomorphism ring is trivial and $\su(3)$ is the stabilizer of an irreducible 3-form with the largest dimension. We also see from the table that the stabilizer of an irreducible 3-form of the second largest dimension has $\su(2)$ as its semisimple part.

Lets consider the algebra $\su(2)\oplus \R^3 \cong \mf u(2)\oplus \R^2$. There is only one faithful Lie algebra representation of this algebra on $\R^8$, namely: $\R^8 = \R^4\oplus \R^2\oplus \R^2$, where $\R^4 = \C^2$ is the vector representation of $\mf u(2)$ and both $\R^2$-summands are an irreducible $\R$-representation. We see that there is no irreducible invariant 3-form for this representation by a similar argument as in \Cref{lem:so(l)<h then no irred 3-form}.
We conclude the biggest dimension of a stabilizer of an irreducible 3-form of dimension less than 8 has dimension less than or equal to 5. So for the case $k=8$ we only have to list all semisimple Lie algebras $\g$ with $\dim(\g)\leq 13$ and add those of dimension $16$.

We see that for $k=7$ there is a stabilizer of an irreducible $3$-form with a relatively large dimension, namely $G_2$. There is only one naturally reductive decomposition which has $\g_2$ as isotropy algebra, the decomposition of $Spin(7)/G_2$. 
In \Cref{tab:7-dim reps} we listed all semisimple Lie algebras with their dimension between 8 and 14 together with all of their 7-dimensional faithful representations. In the third column we indicated if the representation admits an invariant irreducible 3-form.
\begin{table}[ht]
	\[
	\begin{array}{| c | c | c |}
	\hline
	\h & R_\C & \mbox{inv. irred. } 3\mbox{-form}\\
	\hline
	\hline
	\su(3) & R(1,0)\oplus R(0,0) & \text{\cmark}\\
	\hline
	\su(2)^3 & R(1,1,0)\oplus R(0,0,1) & \text{\xmark}\\
	\hline
	\so(5) & R(1,0)\oplus 2R(0,0) & \text{\xmark}\\
	\hline 
	\su(3)\oplus \su(2) & \emptyset & \textrm{n/a}\\
	\hline
	\su(2)^4 & \emptyset & \textrm{n/a}\\
	\hline
	\so(5)\oplus \su(2) & \emptyset &\textrm{n/a}\\
	\hline
	\g_2 & R(1,0) & \text{\cmark}\\
	\hline 
	\su(3)\oplus \su(2)^2 & \emptyset & \textrm{n/a}\\
	\hline
	\end{array}
	\]
	\caption{7-dimensional representations with irreducible 3-forms.}\label{tab:7-dim reps}
\end{table}
\Cref{lem:so(l)<h then no irred 3-form} implies that there doesn't exists an invariant irreducible 3-form for the representations of $\su(2)^3$ and $\so(5)$. The endomorphism ring of the $\su(3)$-representation $R(1,0)\oplus R(0,0)$ is 1-dimensional. We see that the stabilizer of an irreducible 3-form in dimension 7 with the second largest dimension is $\mf u(3)$. For a particular choice of basis in $\R^7$ the $\mf u(3)$-invariant torsion forms are spanned by $e_7\wedge (e_{12}+e_{34}+2e_{56})$,
where $e_{ij}$ denotes $e_i\wedge e_j$.
Thus for $k=7$ we only have to list all semisimple Lie algebras $\g$ with $\dim(\g)\leq 16$ and add to this the pair $(\so(7),\g_2)$. 

\subsection{Classification of type I in dimension 7 \label{subs:type I dim 7}}
Now we follow the classification approach described above in dimension 7.
In the second column of \Cref{tab:candidates dim 7} we list all compact semisimple Lie algebras $\g$ of dimension $7\leq k\leq 16$ and add to this the case $(\g,\h) = (\so(7),\g_2)$. In the third column we list all Lie algebras $\h$ of dimension $\dim(\g)-7$ with
\[
\mbox{rank}(\h)\leq\min\{ \mbox{rank}(\g),\mbox{rank}(\so(7))\} \leq 3.
\]
The following result will exclude many cases from satisfying condition 3.
\begin{lemma}\label{lem:rank equal implies red}
	Let $\g=\g_1\oplus \g_2\oplus \dots \oplus \g_k$, with $\g_i$ simple for $i=1,\dots,k$. Let $\h\subset \g$ be a subalgebra with a naturally reductive decomposition $\g=\h\oplus \m$, where $\m=\h^\perp$ with respect to some $\ad(\g)$-invariant non-degenerate symmetric bilinear form. If $\g=\h\oplus \m$ is irreducible, then 
	\[
	\rank \g \geq  \rank \h + k - 1.
	\]
\end{lemma}
\begin{proof}
	For $k=1$ the statement is true. Suppose that it is true for a certain $k\in \N$. Let $\g = \g_1\oplus \dots \oplus \g_k\oplus \g_{k+1}$ and let us denote $\g'=\g_1\oplus \dots \oplus \g_k$. Let $\pi_1:\g\to \g'$ and $\pi_2:\g \to \g_{k+1}$ be the projections. Let $\h_1:=\ker(\pi_2)$, $\h_3:=\ker(\pi_1)$ and $\h_2\subset \h$ a complementary ideal of $\h_1\oplus \h_3$, which exists because $\h$ is a reductive Lie algebra. Note that $\rank{\h_2}\geq 1$, because otherwise the decomposition is reducible by \Cref{lem:reducible iff ideals}. By our induction hypothesis we have
	\[
	\rank{\g'} \geq \rank{\h_1\oplus \h_2} +k-1.
	\]
	Furthermore, we have $\rank{\g_{k+1}} \geq \rank{\h_2}+\rank{\h_3}$.
	Combining these yields 
	\[
	\rank \g \geq \rank{\h_1} + \rank{\h_2} +k-1 + \rank{\h_2}+\rank{\h_3}\geq \rank\h +k.\qedhere
	\]
\end{proof}

\begin{table}[ht]
	\[
	\begin{array}{|c | c |c| }
	\hline
	\dim(\g) & \g & \h \\
	\hline\hline
	8 & \su(3) & \R \\
	\hline
	9 & \su(2)^3 & \R^2 \\
	\hline 
	10 & \so(5) & \su(2) \\
	\hline
	11 & \su(3)\oplus \su(2) & \su(2)\oplus \R \\
	\hline
	12 & \su(2)^4 & \su(2)\oplus \R^2\\
	\hline
	13 & \so(5)\oplus \su(2) & \su(2)\oplus \su(2)\\
	\hline
	14 & \su(3)\oplus \su(2)^2 & \su(2)^2\oplus \R \\ 
	\hline
	14 & \g_2 & \emptyset \\ 
	\hline
	15 & \su(2)^5 & \su(3)\\
	\hline
	15 & \su(4) & \su(3)\\
	\hline
	16 & \so(5)\oplus \su(2)^2 & \su(3)\oplus \R,~\su(2)^3\\
	\hline
	16 & \su(3)\oplus \su(3) & \su(3)\oplus \R,~\su(2)^3\\
	\hline
	21 & \so(7) & \g_2\\
	\hline
	\end{array}
	\]
	\captionof{table}{Candidates for $7$-dimensional spaces of type I.\label{tab:candidates dim 7}}
\end{table}

Now that we have all candidates for the pairs $(\g,\h)$, it remains to find all possible conjugacy classes of injective Lie algebra homomorphisms $\h\to \g$ such that condition \ref{item torsion irreducible} and \ref{item transvection} from the beginning of this section are satisfied. 
The pairs $(\g,\h)$ which are excluded by \Cref{lem:rank equal implies red} are:
\begin{align*}
&(\su(2)^4,\su(2)\oplus \R^2),\quad (\su(3)\oplus \su(2)^2,\su(2)^2\oplus \R)\quad,(\su(2)^5,\su(3))\\
&(\so(5)\oplus \su(2)^2,\su(3)\oplus \R),\quad (\so(5)\oplus \su(2)^2,\su(2)^3).
\end{align*}
For the pair $(\su(3)\oplus \su(3),\su(2)^3)$ there doesn't exist an injective Lie algebra homomorphism from $\h$ to $\g$. It is easily seen that no injective Lie algebra homomorphism $\su(3)\oplus \R\to \su(3)\oplus \su(3)$ satisfies condition \ref{item torsion irreducible} or \ref{item transvection}. The remaining pairs are
\begin{align*}
& (\su(3),\R),\quad (\su(2)^3,\R^2),\quad (\so(5),\su(2)),\quad (\su(3)\oplus\su(2),\su(2)\oplus \R)\\
& (\so(5)\oplus\su(2),\su(2)\oplus\su(2)),\quad (\su(4),\su(3)),\quad (\so(7),\g_2).
\end{align*}
For these remaining cases we now describe explicitly all subalgebras $\h$ together with all the possible non-degenerate symmetric bilinear forms. 

\vspace{3pt}
\paragraph{\textbf{Case $(\g,\h)=(\su(3),\R)$}}
Every subalgebra $\R\subset \su(3)$ is conjugate to one spanned by
\begin{equation}\label{eq:Aloff-Wallach r(a,b)}
r(a,b) := \begin{pmatrix}
ia & 0 & 0 \\
0 & ib & 0 \\
0 & 0 & -i(a+b)
\end{pmatrix},
\end{equation}
with $a,b\in \R$ and not both equal to zero. By \Cref{lem:equivalent reps} two pairs $(a,b)$ and $(c,d)$ will give an isomorphic infinitesimal model exactly when their subalgebras are conjugate by an element $A\in \Aut(\su(3))$. If $A$ is an inner automorphism, then $A(r(a,b))$ has the same eigenvalues as $r(a,b)$. Therefore $A$ is a signed permutation matrix in $SU(3)$. An outer automorphism $\tau:\su(3)\to \su(3)$ is given by taking the negative transpose in $\su(3)$. We have $\tau(r(a,b))=r(-a,-b)$. The outer automorphism group of $\su(3)$ is $\Z_2$. We can now see that all pairs $(x,y)$ for which $\mbox{span}\{r(x,y)\}$ is conjugate to $\mbox{span}\{r(a,b)\}$ by an automorphism of $\su(3)$ are:
\begin{equation}\label{eq:equivalent pairs aloff-wallach}
\pm(a,b),~ \pm (a,-a-b),~ \pm (b,a),~ \pm(b,-a-b),~ \pm (-a-b,a),~ \pm (-a-b,b).
\end{equation}
By using the above automorphisms we can always arrange that $a\geq b >0$. Thus the isomorphism classes are precisely described by $\frac{a}{b}\geq 1$.
The connected subgroup with this Lie algebra is closed if and only if $\frac{a}{b} = q\in\Q$. The homogeneous spaces are $SU(3)/S^1_{q}$, where $S^1_{q}$ is the image of
\[
S^1\to SU(3);~ \theta\mapsto \begin{pmatrix}
e^{i q \theta} & 0 & 0\\
0 & e^{i\theta } & 0\\
0 & 0 & e^{-i\theta(1+q)}
\end{pmatrix}.
\]
The $\ad(\g)$-invariant non-degenerate symmetric bilinear form $\overline{g}$ on $\g $ is induced from the Killing form of $\su(3)$, hence for every case there is a 1-parameter family of naturally reductive metrics.

\vspace{3pt}
\paragraph{\textbf{Case $(\g,\h) = (\su(2)^3,\R^2)$}}
Let $x_1,x_2,x_3$ be the following basis of $\su(2)$: 
\begin{equation}\label{eq:basis su2}
x_1:=\begin{pmatrix}
i & 0\\
0 & -i
\end{pmatrix},\quad x_2:=\begin{pmatrix}
0 & -1\\
1 & 0
\end{pmatrix},\quad x_3:=\begin{pmatrix}
0 & i\\
i & 0
\end{pmatrix}.
\end{equation}
The $\ad(\g)$-invariant non-degenerate symmetric bilinear form on $\su(2)^3$ is given by $\overline{g} = \frac{-1}{8 \lambda_1^2}B_{\su(2)}\oplus \frac{-1}{8 \lambda_2^2}B_{\su(2)}\oplus \frac{-1}{8 \lambda_3^2}B_{\su(2)}$.
Without loss of generality we assume that $0<\lambda_1\leq \lambda_2\leq \lambda_3$. 
If the naturally reductive decomposition is irreducible, then $\h$ is conjugate by an automorphism of $\su(2)^3$ to a subalgebra spanned by
\[
h_1:=(a_1x_1 , a_2  x_1, 0),\quad h_2:= (0, b_1 x_1,b_2x_1),
\]
with $a_1,a_2,b_1,b_2>0$. If $\lambda_1 = \lambda_2 < \lambda_3$, then $\h$ is conjugate to one with $a_1\leq a_2$. Similarly if $\lambda_1<\lambda_2=\lambda_3$, then we can arrange that $b_1\leq b_2$.
Lastly if $\lambda_1=\lambda_2=\lambda_3$, then we can arrange that $a_1\leq a_2$ and $b_1\leq b_2$.
Under these conditions every irreducible naturally reductive space is exactly represented once. 
The connected subgroup $H$ of $SU(2)^3$ with $\Lie(H)=\h$ is a closed subgroup precisely when $\frac{a_2}{a_1} = q_1 \in \Q\quad \mbox{and}\quad \frac{b_2}{b_1} = q_2\in \Q$.
If $H$ is closed, then it is isomorphic to $S^1\times S^1$. We obtain a 3-parameter family of naturally reductive structures on $SU(2)^3/(S^1_{q_1} \times S^1_{q_2})$, where the parameters are $\lambda_1,\lambda_2,\lambda_3>0$ and $\Lie(S^1_{q_1} \times S^1_{q_2}) = \h$. Note that $(\su(2),\R)$ is a symmetric pair with $(\sl(2,\R),\R)$ its dual symmetric pair. 
We obtain the partial dual spaces by replacing one or two of the $\su(2)$-factors by $\sl(2,\R)$.
If we replace the first factor, then $\overline{g}|_{\m\times \m}$ is positive definite if and only if $\frac{-a_1^2}{\lambda_1^2} + \frac{a_2^2}{\lambda_2^2}<0$ and when we replace the last factor then $\overline{g}|_{\m\times \m}$ is positive definite if and only if $\frac{b_1^2}{\lambda_2^2} - \frac{b_3^2}{\lambda_3^2}<0$. If we replace the middle factor, then the condition becomes $-\frac{\lambda_2^1}{\lambda_1^2}a_1^2b_1^2 - \frac{\lambda_2^2}{\lambda_3^2}b_2^2a_2^2 + \frac{\lambda_2^4}{\lambda_1^2\lambda_3^2}a_1^2b_2^2<0$. We get similar conditions if two out of the three factors are non-compact.

\vspace{3pt}
\paragraph{\textbf{Case $(\g,\h)=(\so(5),\su(2))$}} For this pair there are three inequivalent faithful $5$-dimensional real representations of $\su(2)$. They correspond to the representations $\R^3\oplus \R\oplus \R$, $\R^4\oplus \R$, $\R^5$, where each summand is irreducible. This gives us the following simply connected spaces:
\[
SO(5)/SO(3)_{\mbox{ir}},\quad SO(5)/SO(3)_{\mbox{st}},\quad Sp(2)/Sp(1)_{\mbox{st}},
\]
where $SO(3)_{\mbox{ir}}$ denotes the subgroup given by the 5-dimensional irreducible representation of $SO(3)$, and $SO(3)_{\mbox{st}}$ is the standard $SO(3)$ subgroup of $SO(5)$, and $Sp(1)_{\mbox{st}}\subset Sp(2)$ is the standard $Sp(1)$ subgroup. The first space corresponds to the representation $\R^5$, the second space to $\R^3\oplus \R\oplus \R$ and the last space to $\R^4 \oplus \R$. In particular all the possible infinitesimal models for the pair $(\so(5),\su(2))$ are regular. The metric is induced from the Killing form on $\so(5)$ and thus for each case we get a 1-parameter family of naturally reductive metrics. We can easily see that these three naturally reductive spaces are not isomorphic, because they have pairwise different isotropy representations and the isotropy representations are the same as the holonomy representations of the canonical connections. 

\vspace{3pt}
\paragraph{\textbf{Case $(\g,\h)=(\su(3)\oplus \su(2),\su(2)\oplus \R)$}} Let $f:\h\to \g$ be an injective Lie algebra homomorphism. If $f(\su(2))\subset \su(2)$, then $f(\R)\subset \su(3)$, since $f(\su(2))$ and $f(\R)$ commute. Now condition \ref{item torsion irreducible} and \ref{item transvection} from the beginning of this section are not satisfied. 
There are up to conjugation only two injective Lie algebra homomorphism from $\su(2)$ to $\su(3)$ associated to the irreducible representations on $\C^2$ and $\C^3$. The irreducible representation $\C^3$ defines the irreducible symmetric pair $(\su(3),\so(3))$. This implies that $f(\R)\subset \su(2)$ and thus results in a reducible space, see \Cref{lem:reducible iff ideals}. In other words condition \ref{item torsion irreducible} is not satisfied. Hence the inclusion of $\su(2)$ in $\su(3)$ can only be the standard inclusion. We obtain the following subalgebras:
\[
\su(2)_{\mbox{st}} \oplus \R_{a,b}\subset \su(3)\oplus \su(2)\quad \mbox{ and }\quad \su(2)_{\Delta} \oplus \R\subset \su(3)\oplus \su(2).
\]
In the first inclusion $\su(2)_{\mbox{st}}=i_{\mbox{st}}(\su(2))$ with $i_{\mbox{st}}:\su(2)\to \su(3)$ the standard inclusion, and $\R_{a,b}$ is the subalgebra spanned by
\begin{equation}\label{eq:subalgebra Rab}
\left(\begin{pmatrix}
ia & 0 & 0\\
0 & ia & 0\\
0 & 0 & -2ia
\end{pmatrix}, 
\begin{pmatrix}
ib & 0\\
0 & -ib
\end{pmatrix}\right).
\end{equation}
By \Cref{lem:reducible iff ideals} this naturally reductive decomposition is irreducible if and only if $a$ and $b$ are non-zero. In this case the connected subgroup of $SU(3)\times SU(2)$ with Lie algebra $\su(2)_{\mbox{st}} \oplus \R_{a,b}$ is closed exactly when $\frac{a}{b} = q\in \Q$. 
Hence the infinitesimal model is regular if and only if $\frac{a}{b}\in \Q$. This subalgebra is conjugate by an automorphism of $\su(3)\oplus \su(2)$ to one with $a,b>0$. The $\ad(\g)$-invariant non-degenerate symmetric bilinear form is given by $\overline{g}= \frac{-\lambda_1}{12}B_{\su(3)} \oplus \frac{-\lambda_2}{8}B_{\su(2)}$. For this case $\overline{g}$ has to be positive definite, i.e. $\lambda_1,\lambda_2>0$. We obtain a 2-parameter family of naturally reductive structures on $(SU(3)\times SU(2))/(SU(2)_{\mbox{st}} \times S^1_{q})$, where $\Lie(S^1_{q}) = \R_{a,b}$. 

The subalgebra $\su(2)_\Delta\oplus \R$ is defined by $\su(2)_\Delta := (i_{\mbox{st}}\oplus \id)(\su(2))$ and $\R$ is spanned by
\[
\left(\begin{pmatrix}
i & 0 & 0\\
0 & i & 0\\
0 & 0 & -2i
\end{pmatrix},\begin{pmatrix}
0 & 0\\
0 & 0
\end{pmatrix}\right).
\]
The corresponding naturally reductive decomposition is irreducible and regular. The $\ad(\g)$-invariant non-degenerate symmetric bilinear form is the same as in the previous case. In this case the space can be normal homogeneous or not. The normal homogeneous metrics correspond to $\lambda_1,\lambda_2>0$. For the non-normal homogeneous case we have $\lambda_1>0$, $\lambda_2<0$ and $\lambda_1+\lambda_2<0$. We obtain a 2-parameter family of naturally reductive structures on $(SU(3)\times SU(2))/(SU(2)_\Delta\times S^1)$. 
These spaces are known as Wilking's spaces and are isometric to an Aloff--Wallach space with $q=1$, see \cite{Wilking1999}.

Note that for both cases $(\su(3),f(\su(2)\oplus \R))$ is a symmetric pair. Therefore, by \Cref{rem:recall dual} we see that both spaces have non-compact partial duals. For a non-compact partial dual the $\ad(\g^*)$-invariant non-degenerate symmetric bilinear form is given by $\overline{g}^* = \frac{\lambda_1}{12}B_{\su(2,1)} \oplus \frac{-\lambda_2}{8}B_{\su(2)}$. 
For the first space $\overline{g}^*|_{\h\times \h}$ is negative definite precisely when $-3a^2\lambda_1 + b^2\lambda_2< 0$. For the second space $\overline{g}^*|_{\h\times \h}$ is negative definite if and only if $\lambda_1,\lambda_2>0$ and $-\lambda_1+\lambda_2<0$. 
For the first space also $(\su(2),\mbox{proj}_{\su(2)}(\h))$ is a symmetric pair. If we replace this pair with its symmetric dual we obtain a naturally reductive structure on
\[
(SU(3)\times SL(2,\R))/(SU(2)\times S^1_{q}).
\]
The $\ad(\g^*)$-invariant non-degenerate symmetric bilinear form is $\overline{g}^*=\frac{-\lambda_1}{12}B_{\su(3)}\oplus \frac{\lambda_2}{8}B_{\sl(2,\R)}$. We have $\overline{g}^*|_{\m \times\m}$ is positive definite if and only if $\lambda_1,\lambda_2>0$ and $3a^2\lambda_1 - b^2 \lambda_2<0$. Suppose we replace both factors by their non-compact dual. The invariant symmetric bilinear form is $\frac{\lambda_1}{12}B_{\su(2,1)} \oplus \frac{\lambda_2}{8}B_{\sl(2,\R)}$ with $\lambda_1,\lambda_2>0$. This has signature $(6,5)$ and thus $\overline{g}|_{\m\times\m}$ is never positive definite, which is not allowed.

\vspace{3pt}
\paragraph{\textbf{Case $(\g,\h)=(\so(5)\oplus \su(2),\su(2)\oplus \su(2))$}} 
In order for condition \ref{item torsion irreducible} to be satisfied we see that both $\su(2)$ factors of $\h$ need to have a non-zero image in $\so(5)$. There is only one $5$-dimensional orthogonal faithful representation of $\su(2)\oplus \su(2) \cong \so(4)$ and this corresponds to the standard inclusion of $\so(4)$ in $\so(5)$. We will denote the image of the $\su(2)$-summand which has non-zero image in both $\so(5)$ and $\su(2)$ by $\su(2)_\Delta$. The associated infinitesimal model is always regular and this gives us the following naturally reductive space:
\[
(Spin(5)\times SU(2))/ (SU(2)_\Delta\times SU(2)).
\]
On this homogeneous space we have a 2-parameter family of $\ad(\g)$-invariant non-degenerate symmetric bilinear forms: $\overline{g}:= \frac{-\lambda_1}{6} B_{\so(5)} \oplus \frac{-\lambda_2}{8} B_{\su(2)} $. The normal homogeneous spaces correspond to the parameter $\lambda_1,\lambda_2>0$. The non-normal homogeneous spaces correspond to $\lambda_1>0$, $\lambda_2<0$ and $2\lambda_1 + \lambda_2<0$. The inequality ensures that $\overline{g}|_{\su(2)_\Delta\times \su(2)}$ is negative definite and thus $\overline{g}|_{\m\times \m}$ is positive definite, where $\m$ is the orthogonal complement of $\su(2)_{\Delta}\oplus \su(2)$ in $\mf{spin}(5)\oplus \su(2)$ with respect to $\overline{g}$. This space is known as the squashed $7$-sphere. This is one of the homogeneous spaces for which there exists a proper nearly parallel $G_2$-structure, see \cite{FriedKathMorSem1997}.

Note that $(\so(5),f(\su(2)\oplus \su(2))$ is a symmetric pair. From \Cref{rem:recall dual} we see that there exists a non-compact partial dual. For the non-compact partial dual the $\ad(\g^*)$-invariant non-degenerate symmetric bilinear form is given by $\overline{g}^* = \frac{\lambda_1}{6} B_{\so(4,1)} \oplus \frac{-\lambda_2}{8} B_{\su(2)}$. The parameters $\lambda_1$ and $\lambda_2$ have to satisfy $\lambda_1,\lambda_2>0$ and $-2\lambda_1+\lambda_2<0$ for the metric $g|_{\m\times\m}^*$ to be positive definite.

\vspace{3pt}
\paragraph{\textbf{Case $(\g,\h) = (\su(4),\su(3))$}}
There are two non-equivalent faithful representations of $\su(3)$ on $\C^4$. They correspond to the reducible representations $\C^3\oplus \C=R(1,0)\oplus R(0,0)$ and $\overline{\C^3}\oplus \C = R(0,1)\oplus R(0,0)$. The two subalgebras defined by these representations are conjugate by an outer automorphism of $\su(4)$. Therefore, there is only one injective Lie algebra homomorphism $\su(3)\to \su(4)$ up to conjugation and this is the standard inclusion. This yields the $7$-dimensional Berger sphere as a naturally reductive space
\[
SU(4)/SU(3).
\]
The associated infinitesimal model is always regular and we get a 1-parameter family of metrics. 

\vspace{3pt}
\paragraph{\textbf{Case $(\g,\h) = (\so(7),\g_2)$}}
There is up to conjugation only one subalgebra $\g_2\subset \so(7)$ and the corresponding infinitesimal model is regular. There is only a 1-parameter family of metrics and the corresponding naturally reductive space $SO(7)/G_2$ is isometric to $S^7$ with a round metric.

\subsection{Classification of type I in dimension 8 \label{sec:type I dim 8}}
In the second column of \Cref{tab:candidates dim 8} we list all candidates of compact semisimple Lie algebras $\g$ of dimension $8\leq k\leq 16$. We have already shown that $\g$ can only have dimension less than or equal to 13 or the dimension of $\g$ is 16. In the third column of \Cref{tab:candidates dim 8} we list all Lie algebras of dimension $\dim(\g)-8$ which satisfy $\mbox{rank}(\h)\leq \min(\mbox{rank}(\g),\mbox{rank}(\so(8))\leq 4$.

\begin{table}[h!]
	\[
	\begin{array}{|c | c |c |}
	\hline
	\dim(\g) & \g & \h\\
	\hline\hline
	8 & \su(3) & \{0\} \\
	\hline
	9 & \su(2)^3 & \R\\
	\hline 
	10 & \so(5) & \R^2\\
	\hline
	11 & \su(3)\oplus \su(2) & \su(2), \ \R^3 \\
	\hline
	12 & \su(2)^4 & \su(2)\oplus\R, \R^4 \\
	\hline
	13 & \so(5)\oplus \su(2) & \su(2)\oplus \R^2\\
	\hline
	16 & \so(5)\oplus \su(2)^2 & \su(3),\ \su(2)^2\oplus \R^2 \\
	\hline
	16 & \su(3)^2 & \su(3),\ \su(2)^2\oplus \R^2 \\
	\hline
	\end{array}
	\]
	\captionof{table}{Candidates for $8$-dimensional spaces of type I.\label{tab:candidates dim 8}}
\end{table}

\noindent The pairs $(\g,\h)$ which are excluded by \Cref{lem:rank equal implies red} are:
\begin{align*}
&(\su(3)\oplus \su(2),\R^3),\quad (\su(2)^4,\su(2)\oplus \R),\quad (\su(2)^4,\R^4),\\
&(\so(5)\oplus \su(2),\su(2)\oplus \R^2),\quad (\so(5)\oplus\su(2)^2,\su(2)^2\oplus \R^2),\quad (\su(3)^2,\su(2)^2\oplus \R^2).
\end{align*}
For the pair $(\so(5)\oplus\su(2)^2,\su(3))$ there does not exist an injective Lie algebra homomorphism from $\su(3)$ to $\so(5)\oplus\su(2)^2$.

The remaining cases are:
\begin{align*}
(\su(3),\{0\}), \quad (\su(2)^3,\R), \quad (\so(5),\R^2), \quad (\su(3)\oplus \su(2),\su(2)),\quad (\su(3)^2,\su(3)).
\end{align*}
We will discuss them case by case:

\vspace{3pt}
\paragraph{\textbf{Case $(\g,\h) = (\su(3),\{0\})$}} 
The pair $(\su(3),\{0\})$ is always regular. The simply connected naturally reductive space for this case is $SU(3)$ with some bi-invariant metric. In other words we have a 1-parameter family of naturally reductive structures.

\vspace{3pt}
\paragraph{\textbf{Case $(\g,\h) = (\su(2)^3,\R)$}} 
Let $x_1,x_2,x_3$ be as in \eqref{eq:basis su2}. The $\ad(\g)$-invariant non-degenerate symmetric bilinear form is given by $\overline{g} = \frac{-1}{8 \lambda_1^2}B_{\su(2)}\oplus \frac{-1}{8 \lambda_2^2}B_{\su(2)}\oplus \frac{-1}{8 \lambda_3^2}B_{\su(2)}$ and is necessarily positive definite, so we can assume $0<\lambda_1\leq \lambda_2\leq \lambda_3$.  Every subalgebra $\R\subset \su(2)^3$ is conjugate to one given by
\[
\R_{a_1,a_2,a_3} = \mbox{span}\{(a_1x_1, a_2x_1, a_3x_1)\},
\]
with $a_1,a_2,a_3\geq 0$.
If $\lambda_1 = \lambda_2<\lambda_3$, then we can conjugate $\h$ such that $a_1 \leq a_2$. Similarly if $\lambda_1<\lambda_2 = \lambda_3$, then we can arrange that $a_2\leq a_3$. Lastly if $\lambda_1 = \lambda_2 = \lambda_3$, then we can arrange that $a_1 \leq a_2 \leq a_3$.
Under these conditions none of these are conjugate to each other. 
From \Cref{lem:reducible iff ideals} we see that the naturally reductive decomposition is irreducible if and only if all $a_1,a_2,a_3$ are non-zero. Clearly the connected subgroup of $SU(2)^3$ with this Lie algebra is a closed subgroup if and only if $\frac{a_2}{a_1}=q_1\in \Q$ and $\frac{a_3}{a_1}=q_2\in \Q$. If it is closed, then we obtain a 3-parameter family of naturally reductive structures on $SU(2)^3/S^1_{q_1,q_2}$, where $S^1_{q_1,q_2}$ is the connected subgroup with $\Lie(S^1_{q_1,q_2}) = \h$. Note that $(\su(2),\R)$ is a symmetric pair with $(\sl(2,\R),\R)$ its dual symmetric pair. We obtain the partial dual spaces by replacing one $\su(2)$-factor by $\sl(2,\R)$. If we replaced the $j$th $\su(2)$-summand by $\sl(2,\R)$, then the restriction $\overline{g}^*|_{\h\times \h}$ is negative definite if and only if $\sum_{i=1}^3 (-1)^{\delta_{ij}}(\frac{a_i}{\lambda_i})^2 <0$.

\vspace{3pt}
\paragraph{\textbf{Case $(\g,\h) = (\so(5),\R^2)$}} 
The subalgebra $\R^2\subset \so(5)$ has to be the maximal torus. In particular these spaces are always regular. The simply connected naturally reductive space for this case is $SO(5)/(SO(2)\times SO(2))$, where $SO(2)\times SO(2)$ is embedded block diagonally.
The metric is induced from any negative multiple of the Killing form of $\so(5)$. In other words we have a 1-parameter family of naturally reductive structures.

\vspace{3pt}
\paragraph{\textbf{Case $(\g,\h) = (\su(3)\oplus\su(2),\su(2))$}}
Up to conjugation there are two injective Lie algebra homomorphisms $\su(2)\to \su(3)\oplus \su(2)$ such that condition \ref{item torsion irreducible} and \ref{item transvection} from the beginning of this section are satisfied. For the inclusion in the second factor there is only the identity. For the inclusion in $\su(3)$ there are two choices, namely the standard inclusion, denoted by $i_{\mbox{st}}$ and the other given by the 3-dimensional irreducible representation of $\su(2)$, denoted by $i_{\mbox{ir}}$. For both inclusions the infinitesimal model is regular. The simply connected homogeneous spaces are:
\[
(SU(3)\times SU(2))/(i_{\mbox{st}}\times \mbox{id})(SU(2))\quad \mbox{and}\quad (SU(3)\times SU(2))/(i_{\mbox{ir}}\times \mbox{id})(SU(2)),
\]
where we denote the corresponding group homomorphism of $i_{\mbox{st}}$ and $i_{\mbox{ir}}$ also by $i_{\mbox{st}}$ and $i_{\mbox{ir}}$, respectively. There is a 2-parameter family of $\ad(\g)$-invariant non-degenerate symmetric bilinear forms: $\overline{g}=\frac{-\lambda_1}{12} B_{\su(3)}\oplus \frac{-\lambda_2}{8} B_{\su(2)}$. The normal homogeneous spaces correspond to $\lambda_1,\lambda_2>0$. For the non-normal homogeneous spaces we have $\lambda_1>0$ and $\lambda_2<0$. Furthermore, we require that the condition $\lambda_1+\lambda_2<0$ holds for the first space and $4\lambda_1 + \lambda_2 < 0$ for the second space.

For the space $(SU(3)\times SU(2))/(i_{\mbox{ir}}\times \mbox{id})(SU(2))$ there is a non-compact partial dual space
\[
(SL(3,\R)\times SU(2))/(i_{\mbox{ir}}\times \mbox{id})(SU(2)).
\]
The $\ad(\g^*)$-invariant non-degenerate symmetric bilinear forms are $\overline{g}^* = \frac{\lambda_1}{12}B_{\sl(3,\R)} \oplus \frac{-\lambda_2}{8}B_{\su(2)}$.
In order to obtain a positive definite metric on our space require that $\lambda_1,\lambda_2>0$ and $-4\lambda_1 + \lambda_2 <0$.

\vspace{3pt}
\paragraph{\textbf{Case $(\g,\h) = (\su(3)^2,\su(3))$}} 
There are two possible conjugacy classes of the subalgebra $\su(3)$, namely $\su(3)\times\{0\}$ and the diagonal $\su(3)_\Delta$. The first case clearly doesn't satisfy condition \ref{item transvection}. 
Therefore, the subalgebra $\h$ has to be the diagonal subalgebra.
The $\ad(\g)$-invariant metrics are given by $\overline{g}=-\lambda_1 B_{\su(3)}\oplus -\lambda_2 B_{\su(3)}$, with $\lambda_1\neq 0$ and $\lambda_2\neq 0$. By permuting the two $\su(3)$-factors we can assume that $\lambda_1>\lambda_2$. The normal homogeneous spaces correspond to $\lambda_1,\lambda_2>0$. Note that for $\lambda_1=\lambda_2$ and $\lambda_1>0$ we obtain a symmetric space. For the non-normal homogeneous spaces we require that the signature of $\overline{g}$ is $(8,8)$ and that $\overline{g}|_{\h\times \h}$ is negative definite. This is the case if and only if $\lambda_1+\lambda_2<0$ and $\lambda_1>0>\lambda_2$. All the naturally reductive structures are regular and irreducible. For every case the homogeneous space is isometric to $SU(3)$ with some bi-invariant metric. 

The pair $(\g,\h)$ is a symmetric pair. The $\ad(\g^*)$-invariant non-degenerate symmetric bilinear forms for the dual pair $(\sl(3,\C),\su(3))$ are all a multiple of the Killing form of $\sl(3,\C)$ and thus all induce a symmetric structure. Consequently, there are no non-symmetric partial dual naturally reductive structures. 

\vspace*{1em}

This concludes the classification of all 7- and 8-dimensional naturally reductive spaces of type I. We summarize the discussion from \Cref{subs:type I dim 7} and \Cref{sec:type I dim 8} as the following result.

\begin{theorem}\label{thm:classification type I}
	Every 7- and 8-dimensional compact simply connected globally homogeneous naturally reductive space of type I is presented in \Cref{tab:7-dim type I}. In the first column $\Lie(G)$ is the transvection algebra of the naturally reductive space. The second column indicates if there exist non-compact partial dual naturally reductive spaces. The third column indicates the number of parameters of naturally reductive structures. 
\end{theorem} 
\begin{center}
	\begin{table}[h!]
		\[
		\begin{array}{ | c | c | c | }
		\hline
		G/H & \mbox{dual} & \#~\mbox{param.} \\
		\hline\hline
		SU(3)/S^1_{q} & \text{\xmark} & 1 \\
		\hline
		SU(2)^3/(S^1_{q_1}\times S^1_{q_2}) & \text{\cmark} & 3 \\
		\hline
		SO(5)/SO(3)_{\mbox{ir}} & \text{\xmark} & 1 \\
		\hline
		SO(5)/SO(3)_{\mbox{st}} & \text{\xmark} & 1 \\
		\hline
		Sp(2)/Sp(1)_{\mbox{st}} & \text{\xmark} & 1 \\
		\hline
		(SU(3)\times SU(2))/(SU(2)\times S^1_{q}) & \text{\cmark}  & 2  \\
		\hline
		(SU(3)\times SU(2))/(SU(2)_{\Delta}\times S^1) & \text{\cmark}  & 2  \\
		\hline
		(Spin(5)\times SU(2))/(SU(2)_\Delta\times SU(2)) & \text{\cmark} & 2 \\
		\hline
		SU(4)/SU(3) & \text{\xmark} & 1 \\
		\hline
		Spin(7)/G_2 & \text{\xmark} & 1 \\
		\hline\hline
		SU(3) & \text{\xmark} & 1 \\
		\hline
		SU(2)^3/S^1_{q_1,q_2} & \text{\cmark} & 3 \\
		\hline
		SO(5)/(SO(2)\times SO(2)) & \text{\xmark} & 1 \\
		\hline
		(SU(3)\times SU(2))/SU(2)_{\mbox{st}\times \mbox{id}} & \text{\xmark} & 2 \\
		\hline
		(SU(3)\times SU(2))/SU(2)_{\mbox{ir}\times \mbox{id}} & \text{\cmark} & 2 \\
		\hline
		(SU(3)\times SU(3))/SU(3)_\Delta & \text{\cmark} & 2\\ 
		\hline
		\end{array}
		\]
		\captionof{table}{7- and 8-dimensional naturally reductive spaces of type I.\label{tab:7-dim type I}}
	\end{table}
\end{center}

\section{Classification of type II \label{sec:class II}}

By Theorem \ref{thm:canonical base} we can construct every infinitesimal model of a naturally reductive decomposition of type II as a $(\mf k,B)$-extension of a naturally reductive decomposition of the form
\[
\g= \h\oplus \m\oplus_{L.a.} \R^n,
\]
where $\h\oplus \m$ is semisimple and $\g$ is the transvection algebra of this naturally reductive decomposition. In this section we will construct all 7 and 8 dimensional irreducible $(\mf k,B)$-extensions of all naturally reductive decomposition of the above form with $\h\oplus \m$ compact. We use the partial duality to obtain all other spaces, see \Cref{rem:recall dual}.
For every case we will mention if there exist partial dual spaces. 
\begin{definition}
	For a canonical base space $\g = \h\oplus \m\oplus_{L.a.}\R^n$ we will call $\h\oplus \m$ the \emph{semisimple factor} and $\R^n$ the \emph{Euclidean factor}.
\end{definition}
We start by finding all possible candidates for the canonical base spaces of irreducible type II spaces. From this list we construct all possible irreducible $(\mf k,B)$-extensions. To guarantee there are no duplicates in our list we use \Cref{lem:canonical base iff} and \Cref{prop:iso type II}.

Note that to classify the naturally reductive spaces of type II in some dimension $k$ we need the classification of all naturally reductive spaces of type I up to dimension $k-1$.

\begin{remark}\label{rem:reducible extension} 
	We want all of our $(\mf k,B)$-extensions to be irreducible. If there exists an irreducible $(\mf k,B)$-extension of a naturally reductive transvection algebra as in \eqref{eq:irreducible factor}, then \Cref{lem:reducibility criteria} in particular implies $\mf s(\h_i\oplus \m_i)\neq \{0\}$ for every $i =1,\dots,p$. In particular this excludes the possibility that $\h_i\oplus \m_i$ is an irreducible symmetric decomposition which is not hermitian symmetric.
	
	If $\mf k$ is abelian and $\mf k = \mf k_1$, then by \Cref{lem:canonical base iff} condition $(i)$ we require that $\pi_\m(\mc Z(\b_1)) = \{0\}$. We need this condition in order for the canonical base space to be the base space we start with. Note that $\pi_\m(\mf k_1) =\{0\}$ if and only if $\mf k_1\subset \mc{Z}(\h)$. Thus for the $(\mf k,B)$-extension to be irreducible and satisfy condition $(i)$ from \Cref{lem:canonical base iff} we require that $\mc Z(\h_i)\neq \{0\}$ for every $i =1,\dots,p$.
\end{remark}

\subsection{Classification of type II in dimension 7\label{sec:type II dim 7}}
First we argue that all possible canonical base spaces of irreducible naturally reductive decompositions of type II with a compact semisimple factor are given in \eqref{eq:canonical base 7}. This is done by systematically excluding all other possibilities. 
\begin{equation}\label{eq:canonical base 7}
\begin{array}{l l l}
\R^6, & \R^4, & S^2\times \R^4, \\
\C P^2\times \R^2, & S^2\times S^2\times \R^2, &  Sp(2)/(SU(2)\times S^1),\\
SO(5)/(SO(3)\times SO(2)), & SU(3)/(S^1\times S^1), & SU(4)/S(U(1)\times U(3)),\\
\C P^2\times S^2, & S^2\times S^2\times S^2. &
\end{array}
\end{equation}
\noindent Even though we write all above base spaces as globally homogeneous spaces we actually treat the family of naturally reductive decompositions to which they belong, which can also contain non-regular decomposition, i.e. strictly locally homogeneous spaces.
It has to be considered case by case for which parameter values the locally naturally reductive structures are regular. 

The Euclidean factor can't be $\R^5$, because then the Lie algebra $\mf k\subset \so(5)$ is two dimensional and its linear action on $\R^5$ has a vector on which it acts trivially. From \Cref{lem:reducibility criteria} we see that such any $(\mf k,B)$-extension results in a reducible naturally reductive space.

Suppose that the Euclidean factor is $\R^3$, the Lie algebra $\mf k\subset \so(3)$ has to be equal to $\so(3)$ in order not to have a vector on which it acts trivially. This means that the semisimple factor of the base space has to be 1-dimensional, which is not possible.

Suppose that the Euclidean factor is $\R^2$. If the dimension of the semisimple factor is two. Then $\dim(\s(\g))\leq 2$ and thus we can't construct an irreducible $(\mf k,B)$-extension of dimension 7. 
If the dimension of the semisimple factor is three, then $\dim(\mf k) = 2$ and thus $\mf k$ is abelian. The semisimple factor is either $SU(2)$ or the symmetric space $(SU(2)\times SU(2))/SU(2)_\Delta$. 
The last case is excluded by \Cref{rem:reducible extension}. For $SU(2)$ the algebra $\mf k_1$ is $1$-dimensional and thus we see that condition $(i)$ from \Cref{lem:canonical base iff} can not be satisfied. 
If the dimension of the semisimple factor is four, then the semisimple factor has to be a hermitian symmetric space by \cite{KowalskiVanhecke1983}. There are only two compact homogeneous spaces which allow a hermitian symmetric structure, these are $S^2\times S^2$ and $\C P^2$. 

For all other 7-dimensional naturally reductive spaces of type II the base space has only a semisimple factor. It is easy to check that every 7-dimensional $(\mf k,B)$-extension of any naturally reductive space of type I of dimension less than or equal to 4 is reducible. This leaves us with the 5- and 6-dimensional cases. The only compact spaces of type I in dimension 5 with $\dim(\s(\g))\geq 2$ are $S^2\times SU(2)$ and $(SU(2)\times SU(2))/S^1$. 
However, we see for any $2$-dimensional $\mf k\subset \s(\g)$ that condition $(i)$ of \Cref{lem:canonical base iff} is not satisfied in both cases and thus they are excluded. 
The nearly Kähler spaces $G_2/SU(3)$ and $(SU(2)\times SU(2)\times SU(2))/SU(2)_{\Delta}$ can be excluded, because for both $\s(\g) = \{0\}$ holds. Similarly $((SU(2)\times SU(2))/SU(2)_\Delta)\times ((SU(2)\times SU(2))/SU(2)_\Delta)$ satisfy $\s(\g)=\{0\}$. The spaces $SU(2)\times ((SU(2)\times SU(2))/SU(2)_\Delta$ and $SU(2)\times SU(2)$ can be excluded, because they don't satisfy condition $(i)$ from \Cref{lem:canonical base iff} for any $\mf k\subset \s(\g)$. All other 6-dimensional naturally reductive spaces of type I are possible. 

A classification of naturally reductive decompositions of type II in dimension $7$ is readily obtained in the following steps. From the list of possible canonical base spaces in \eqref{eq:canonical base 7} we have to make all irreducible $(\mf k,B)$-extensions such that condition $(i)$ and $(ii)$ of \Cref{lem:canonical base iff} are satisfied. \Cref{lem:reducibility criteria} tells us exactly when the constructed spaces are irreducible.
We also have to filter out all isomorphic spaces. \Cref{prop:iso type II} makes this quite easy in all the occurring cases.
How to obtain a globally homogeneous naturally reductive space from these data is described in \cite{Storm2018}. 
To make the classification complete we just need to check for every case if partial dual naturally reductive spaces exist.
We will not discuss every case because some cases are very similar. We attempt to cover all the different `types' of $(\mf k,B)$-extensions in the selected cases below.
The classification list can be found in \Cref{thm:classification type II}.

\begin{remark}
	From \Cref{lem:ker(R) characterisation} we know that $\mf k_1 =\{0\}$ implies $\ker(R|_{\ad(\h)+\psi(\mf k)}) = \{0\}$. Thus in particular condition $(i)$ and $(ii)$ in \Cref{lem:canonical base iff} are automatically satisfied. Therefore, we will only check the conditions of \Cref{lem:canonical base iff} when $\mf k_1\neq \{0\}$.
\end{remark}

Before we start let us introduce some notation.
\begin{notation}
	Below $e_{ij}\in \so(n)$ is the matrix whose only non-zero entries are its $ij$th and $ji$th entry, which are $-1$ and $1$, respectively. Let $B_{\Lambda^2}$ be the metric on $\so(n)$ be defined by $B_{\Lambda^2}(x,y) := -\frac{1}{2}\mathrm{tr}(xy)$. In the following we use the contraction with the metric on $\m$ to make the identification $\Lambda^2\m\cong \so(\m)$, i.e. $e_{ij}$ is identified with $e_i\wedge e_j$. The curvature tensor then becomes a symmetric map $R:\so(\m)\to \so(\m)$ with respect to $B_{\Lambda^2}$.
\end{notation}

The representation $\varphi:\mf k \to \so(\m)$, from \Cref{def:k_i and b_i and varphi}, uniquely determines the algebra $\mf k \subset \mf s(\g)$ and below we will always describe $\mf k$ through $\varphi(\mf k)$. 

\vspace{3pt}
\paragraph{\textbf{The canonical base space is} $\R^6$}
Consider the canonical base space $\R^6$. The Lie algebra $\mf k$ is 1-dimensional. Let $k$ be a unit vector in $\mf k$. Then there is an orthonormal basis $e_1,\dots,e_6$ of $\R^6$ and constants $c_1,c_2,c_3 \in \R$ such that
\[
\varphi(k) = c_1 e_{12} + c_2 e_{34} + c_3 e_{56}\in \so(6), 
\]
It is clear from \Cref{lem:reducibility criteria} that the spaces are irreducible precisely when $c_1,c_2,c_3 \in \R\backslash \{0\}$. Therefore, from now on we suppose that $c_1,c_2,c_3 \in \R\backslash \{0\}$. The $(\mf k,B)$-extensions describe naturally reductive structures on the 7-dimensional Heisenberg group, as explained in \cite{Storm2018}. We get a 3-parameter family of naturally reductive structures on the 7-dimensional Heisenberg group. We can ensure that $0<c_1\leq  c_2\leq  c_3$ by choosing a different basis of $\R^6$. When we do this, all the described naturally reductive structures are non-isomorphic. In \cite{Kaplan1983} it is proven that the Heisenberg groups and the quaternionic-Heisenberg groups are the only groups of type $H$ for which the natural left invariant metric is naturally reductive.

\vspace{3pt}
\paragraph{\textbf{The canonical base space is} $\R^4$} The Lie algebra $\mf k$ has to be $\su(2)$ and the representation $\varphi:\su(2)\to \so(4)$ has to be the irreducible $4$-dimensional representation in order for the $(\mf k,B)$-extension to be irreducible. The $(\mf k,B)$-extension will yield a naturally reductive structure on the quaternionic Heisenberg group. The choice of an invariant metric $B$ on $\mf k$ gives us a 1-parameter family of naturally reductive structures. 
This family of naturally reductive structures is quite interesting and is investigated in \cite{AgricolaFerreiraStorm15}.

\vspace{3pt}
\paragraph{\textbf{The canonical base space is} $S^2\times \R^4$}
Let $h,e_1,e_2$ be an orthonormal basis of $\su(2)$ with respect to $\frac{-1}{8\lambda_1^2}B_{\su(2)}$. The transvection algebra of the base space is given by
\[
\g = \su(2) \oplus_{L.a.} \R^4 = \h \oplus \m \oplus_{L.a.} \R^4,
\]
where $\h:=\mbox{span}\{h\}$ and $\m:=\mbox{span}\{e_1,e_2\}$. The $\ad(\g)$-invariant non-degenerate symmetric bilinear form on $\g$ is given by $\overline{g} = \frac{-1}{8\lambda_1^2}B_{\su(2)}\oplus B_{eucl}$. We have $\s(\g) = \mbox{span}\{h\}\oplus \so(4)$. Let $k\in \mf k$ be a unit vector. Then there is an orthonormal basis of $\R^4$ such that
\[
\varphi(k) = c_1 \ad(h)|_\m + c_2 e_{34} + c_3 e_{56},
\]
with $c_1,c_2,c_3\in \R$. All these spaces are irreducible precisely when $c_1,c_2,c_3\in \R\backslash\{0\}$ by \Cref{lem:reducibility criteria}. Therefore, from now on we suppose that $c_1,c_2,c_3\in \R\backslash\{0\}$. We have $\mf k=\mf k_2$ and from \cite[Sec.~2.3]{Storm2018} we know that the $(\mf k,B)$-extension defines a naturally reductive structure on $S^2\times H^5$, where $H^5$ denotes the 5-dimensional Heisenberg group. On this homogeneous space we obtain a 4-parameter family of naturally reductive structures, with $c_1,c_2,c_3$ and $\lambda_1>0$ as parameters. By an automorphism of $\g$ we can arrange that $c_2\geq c_3>0$ and $c_1>0$. When we do this, none of these naturally reductive structures are isomorphic. Note that we can replace the semisimple factor $S^2=SU(2)/S^1$ by its non-compact dual symmetric space: $SL(2,\R)/S^1$.

\paragraph{\textbf{The canonical base space is} $Sp(2)/(SU(2)\times S^1)$} We consider $Sp(2)\subset GL(2,\H)$. We denote by $i,j,k$ the imaginary quaternions, i.e. $i^2=j^2=k^2=ijk = -1$ and we pick the following basis for $\mf u(2)$:
\paragraph{\textbf{The canonical base space is} $SU(4)/S(U(1)\times U(3))$:} Because the base space is an irreducible hermitian symmetric space this case is completely analogous to the previous case: $SO(5)/(SO(3)\times SO(2))$. All in all we obtain a 2-parameter family of naturally reductive structures on $SU(4)/SU(3)$ with the standard embedding of $SU(3)$. Note that we can replace $SU(4)/S(U(1)\times U(3))$ by its dual symmetric space: $SU(1,3)/S(U(1)\times U(3))$.

\vspace{3pt}
\paragraph{\textbf{The canonical base space is} $\C P^2 \times S^2$} 
Let $h_1,h_2,h_3,h_4,e_1,e_2,e_3,e_4$ be an orthonormal basis of $\su(3)$ with respect to $\frac{-1}{12\lambda_1^2}B_{\su(3)}$ such that $h_1,h_2,h_3,h_4$ span the Lie algebra of the isotropy group $S(U(2)\times U(1))\subset SU(3)$ with $h_4$ spanning the center.
Let $h_5,e_5,e_6$ be an orthonormal basis of $\su(2)$ with respect to $\frac{-1}{8\lambda_2^2}B_{\su(2)}$. The transvection algebra of the base space is given by $\g = \su(3)\oplus \su(2) = \h \oplus \m$, 
where $\h := \mbox{span}\{h_1\.h_5\}$ and $\m:=\mbox{span}\{e_1\.e_6\}$. The $\ad(\g)$-invariant non-degenerate symmetric bilinear form is $\overline{g} =\frac{-1}{12\lambda_1^2}B_{\su(3)}\oplus \frac{-1}{8\lambda_2^2}B_{\su(2)}$. The algebra $\mf k\subset \s(\g) = \mbox{span}\{h_4,h_5\}$ is 1-dimensional. Let $k\in \mf k$ be a unit vector. Then $\varphi(k) = c_1 \ad(h_4)|_\m + c_2 \ad(h_5)|_\m$.
The curvature of the $(\mf k,B)$-extension is given by
\[
R = -\sum_{i=1}^5 \ad(h_i)|_\m\odot \ad(h_i)|_\m + \varphi(k)\odot \varphi(k).
\]
From \Cref{lem:ker(R) characterisation} we have $\ker(R|_{\ad(\h\oplus\mf k)}) = \ker(R|_{\ad(\mc Z(\h\oplus \mf k)}) = \ker(R|_{\ad(\mc Z(\h)})$. We need to check when $R|_{\ad(\mc Z(\h))}$ has trivial kernel. Note that the center of $\h$ is given by $ \mbox{span}\{h_4,h_5\}$. Let $\omega_1,\omega_2\in \so(\m)$ be such that $B_{\Lambda^2}(\omega_1, h_j) = \delta_{4j}$ and $B_{\Lambda^2}(\omega_2, h_j) = \delta_{5j}$ for $j=1\.5$. Then
\begin{align*}
R(\omega_1)&=(-1+c_1^2) \ad(h_4)|_\m +c_1c_2 \ad(h_5)|_\m,  \\
R(\omega_2)&=c_1c_2\ad(h_4)|_\m + (-1+c_2^2)\ad(h_5)|_\m.
\end{align*}
We see that $R|_{\ad(\mc Z(\h))}$ has rank 2 precisely when $c_1^2+c_2^2\neq 1$. In other words the base space is equal to the canonical base space if and only if $c_1^2+c_2^2\neq 1$. By \Cref{lem:reducibility criteria} the $(\mf{k},B)$-extension is reducible precisely when either $c_1=0$ or $c_2=0$. Suppose that the $(\mf{k},B)$-extension is irreducible. With an automorphism of $\g$ we can always arrange that $c_1>0$ and $c_2>0$. Under this condition none of the described $(\mf{k},B)$-extensions are isomorphic. The $(\mf{k},B)$-extension is regular if and only if the connected subgroup $H_0$ with Lie subalgebra $\h_0=\mf{k}^\perp\subset \h$ is closed in $SU(3)\times SU(2)$, see \cite{Storm2018}. We have $\h_0 = \mbox{span}\{c_2 h_4 - c_1 h_5\}$. We see that $H_0$ is closed precisely when $q=\frac{c_2\lambda_1}{\sqrt{3}c_1\lambda_2}\in \Q$. The $(\mf{k},B)$-extension describes a naturally reductive structure on $(SU(3)\times SU(2))/(SU(2)\times S^1_{q})$, where $SU(2)$ is the standard subgroup of $SU(3)$ and $S^1_{q}$ is the subgroup with Lie subalgebra $\h_0$. To obtain all of these naturally reductive structures on the fixed homogeneous space $(SU(3)\times SU(2))/(SU(2)\times S^1_{q})$ we start by defining $\h_0:=\Lie(S^1_{q})$ and $\mf{k}:=\h_0^\perp\subset \h$ with respect to $\overline{g}$. We have a 1-parameter family of $\ad(\mf{k})$-invariant metrics on $\mf{k}$. Together with the parameters $\lambda_1,\lambda_2$ this gives us a 3-parameter family of naturally reductive structures on $(SU(3)\times SU(2))/(SU(2)\times S^1_{q})$.  Note that we can replace $SU(3)/S(U(2)\times U(1))$ by its symmetric dual $SU(2,1)/S(U(2)\times U(1))$ and we can also replace $S^2$ by its non-compact dual.

\subsection{Classification of type II in dimension 8\label{sec:type II dim 8}}
First we argue that all possible canonical base spaces of irreducible naturally reductive decompositions of type II with a compact semisimple factor are given in \eqref{eq:canonical base 8}. This is done by systematically excluding all other possibilities. 
%
\begin{equation}\label{eq:canonical base 8}
\begin{array}{l l}
\R^6, & \R^5,\\
\R^4, & S^2\times \R^4\\
SU(2)\times \R^4, & \C P^2\times \R^2, \\
S^2\times S^2\times \R^2, &(SU(2)\times SU(2))/S^1_{q}\times \R^2, \\
SU(3)/SU(2)_{\mbox{st}}\times\R^2, & SU(2)\times S^2\times \R^2,  \\
SU(3)/S^1_{q}, & (SU(3)\times SU(2))/(SU(2)_{\mbox{st}}\times S^1_{q}), \\
(SU(3)\times SU(2))/(SU(2)_\Delta \times S^1), & SU(2)^3/(S^1_{q_1}\times S^1_{q_2}), \\
S^2\times (SU(2)\times SU(2))/S^1_{q}, & SU(3)/(S^1\times S^1), \\
S^2\times S^2\times S^2, & S^2 \times \C P^2\\
\{*\}. & 
\end{array} 
\end{equation}
\noindent where $\{*\}$ denotes a point space. Even though we write all above base spaces as globally homogeneous spaces we actually treat the family of naturally reductive decompositions to which they belong, which can also contain non-regular decomposition, i.e. strictly locally homogeneous spaces.
It has to be considered case by case for which parameter values the locally naturally reductive structures are regular.

The Euclidean factor can't be $\R^7$, because then $\dim(\mf k)=1$ and the linear action of $\mf k$ on $\R^7$ has a vector on which it acts trivially and by \Cref{lem:reducibility criteria} any such $(\mf k,B)$-extension is reducible.

If the Euclidean factor is $\R^6$, then the semisimple factor needs to have dimension zero and $\dim(\mf k)=2$.

If the Euclidean factor is $\R^5$ and the semisimple factor is 2-dimensional, then $\dim(\mf k)=1$. Just as for $\R^7$ we see that the linear action of $\mf k$ on $\R^5$ has a vector on which it acts trivially and by \Cref{lem:reducibility criteria} any such $(\mf k,B)$-extension is reducible. Thus, also for $\R^5$ the semisimple factor has to be zero dimensional.

Suppose that the Euclidean factor is $\R^4$. The semisimple factor can be 0-, 2- or 3-dimensional. If it is 2-dimensional, then it is $S^2$. If it is 3-dimensional, then it either is the symmetric space $(SU(2)\times SU(2))/SU(2)$ or the Lie group $SU(2)$. The first case is excluded by \Cref{rem:reducible extension}.

If the Euclidean factor is $\R^3$, then $\mf k$ has to contain $\so(3)$ in order for the linear representation of $\mf k$ on $\R^3$ not to have a vector on which it acts trivially. We see that if the semisimple factor is 0-dimensional, then we can't construct an irreducible 8-dimensional $(\mf k,B)$-extension. The only other possibility is that the semisimple factor is 2-dimensional. In this case we immediately see by \Cref{lem:reducibility criteria} that any such $(\so(3),B)$-extension is reducible.

Suppose that the Euclidean factor is $\R^2$. The semisimple factor can either be 3-,4- or 5-dimensional, because if the semisimple factor is 2-dimensional, then $\dim(\s(\g))\leq 2$ and thus we can't make an irreducible 8-dimensional $(\mf k,B)$-extension from this. Suppose that the semisimple factor is 5-dimensional. 
We see that there are three possibilities: $(SU(2)\times SU(2))/S^1_{q}$, $SU(3)/SU(2)$ and $SU(2)\times S^2$. 
Suppose that the semisimple factor is 4-dimensional. If it is irreducible, then it can only be $\C P^2$. If it is reducible, then it can only be $S^2\times S^2$. Suppose that the semisimple factor is 3-dimensional. From \Cref{rem:reducible extension} we see that the semisimple factor has to be equal to $SU(2)$ and $\s(\g) = \su(2)\oplus \so(2)$. The Lie algebra $\mf k\subset \s(\g)=\su(2)\oplus \so(2)$ is a 3-dimensional subalgebra. Hence $\mf k=\su(2)\subset \s(\g)$ and thus $\mf k$ acts trivially on $\R^2$. Therefore,  by \Cref{lem:reducibility criteria}, any such $(\mf k,B)$-extension is reducible. 

Only base spaces with no Euclidean part remain. 
Now we discuss the case for which the base space has an irreducible 3-dimensional factor. There are only two compact irreducible 3-dimensional naturally reductive spaces of type I: $SU(2)$ and the symmetric space $(SU(2)\times SU(2))/SU(2)$. The symmetric space is excluded by \Cref{rem:reducible extension}. If we have $SU(2)$ as a 3-dimensional factor, then $\mf k$ has to be at least 3-dimensional, see \Cref{lem:canonical base iff} condition $(i)$. The only possibility for a base space is $SU(2)\times S^2$, but just as for the case $SU(2)\times \R^2$ any 8-dimensional $(\mf k,B)$-extension of this space is reducible. We conclude that if there is no Euclidean factor, then the semisimple factor can not contain a 3-dimensional factor.

If the base space is 7-dimensional, then $\dim(\mf k)=1$ and thus $\mf k$ is abelian. By \Cref{rem:reducible extension} we require that $\mc{Z}(\h_i)\neq \{0\}$ for every $i=1,\dots,p$. We noted above that there can't be a 3-dimensional factor, hence the 7-dimensional space either is irreducible or it is a product of a 5-dimensional irreducible space and a 2-dimensional space. Consequently, all possible spaces are: $SU(3)/S^1_{q},~(SU(3)\times SU(2))/(SU(2)\times S^1_{q}),~(SU(3)\times SU(2))/(SU(2)_\Delta\times S^1),~SU(2)^3/(S^1_{q_1}\times S^1_{q_2})$ and $(SU(2)\times SU(2))/S^1_{q}\times S^2$. 

For a 6-dimensional base space $\mf k$ is abelian and thus by \Cref{rem:reducible extension} we need that $\mc{Z}(\h_i)\neq \{0\}$ for every $i=1,\dots,p$. There are no $3$-dimensional factors by the above argument. We can easily see that all possibilities are: $SU(3)/(S^1\times S^1),~\C P^2\times S^2$ and $S^2\times S^2\times S^2$. 

We check that every 5-dimensional irreducible naturally reductive space of type I satisfies $\dim(\s(\g))\leq 2$
and thus we can't make an 8-dimensional irreducible $(\mf k,B)$-extension from this. Every reducible 5-dimensional space of type I contains a 3-dimensional factor and thus can be excluded by the above discussion. Similarly for every 4-dimensional space of type I we have $\dim(\s(\g))\leq 2$ and thus we can not make an irreducible 8-dimensional $(\mf k,B)$-extension of this.

The Lie algebra $\su(3)$ has dimension 8 and is a compact simple Lie algebra. Therefore, also a point space is a possible base space.

We proceed as in the $7$-dimensional case. Also here we do not discuss every case separately because of the large similarities between them. 

\vspace{1em}

\vspace{3pt}
\paragraph{\textbf{The canonical base space is $\R^5$}}
The Lie algebra $\mf k$ has to be 3-dimensional and in order to have a 5-dimensional representation without vectors on which $\mf k$ acts trivially. The only possibility is $\mf k = \su(2)$ and the representation of $\mf k$ is the 5-dimensional irreducible representation of $\su(2)$. Let $k_1,k_2,k_3$ be an orthonormal basis of $\su(2)$ with respect to $B = -\frac{1}{2\lambda^2} B_{\su(2)}$. We choose a basis such that
\[
\varphi(k_1) = \lambda(\sqrt{3}e_{13} -e_{24} -e_{35}),~ \varphi(k_2) = \lambda(-\sqrt{3}e_{12} +e_{34}-e_{25}),~ \varphi(k_3) = \lambda(e_{23}+2e_{45}).
\]
The $(\mf k,B)$-extension defines a naturally reductive structure on an 8-dimensional 2-step nilpotent Lie group, as described in \cite[Sec.~2.2]{Storm2018}. On this homogeneous space we obtain a 1-parameter family of naturally reductive structures, with $\lambda>0$ as parameter.

\vspace{3pt}
\paragraph{\textbf{The canonical base space is $(SU(2)\times SU(2))/S^1_\alpha\times \R^2$}}
The Lie algebra $\mf k$ is 1-dimensional. Let $k\in \mf k$ be a unit vector. To keep the notation concise we consider $\su(2)\cong \sp(1)\subset \mf{gl}(1,\H)$. We denote by $i,j,k$ the imaginary quaternions, i.e. $i^2=j^2=k^2=ijk = -1$.
The non-degenerate symmetric bilinear form on $\sp(1)\oplus \sp(1)$ is given by $-\frac{1}{8 \lambda_1^2} B_{\sp(1)}\oplus -\frac{1}{8 \lambda_2^2} B_{\sp(1)}$, where $B_{\sp(1)}$ denotes the Killing form of $\sp(1)$. Let
\begin{align}\label{eq:basis for su(2)x su(2)/S1}
e_1&:=(\lambda_1 j,0),\quad e_3:=(0,\lambda_2 j), \quad e_5:= (\alpha^2\lambda_1^2 + \lambda_2^2)^{-1/2}\left(\lambda_1^2 \alpha 
i,-\lambda_2^2 i\right),  \notag\\
e_2&:=(\lambda_1k,0),\quad e_4:=(0,\lambda_2 k),\quad h:=\frac{\lambda_1\lambda_2}{\sqrt{\alpha^2\lambda_1^2 + \lambda_2^2}}(i,\alpha i),
\end{align}
where $e_1,\dots,e_5$ is an orthonormal basis of $\m:=h^\perp$ with respect to the metric above and $\alpha\in \R\backslash\{0\}$. Let $\{e_6,e_7\}$ be an orthonormal basis of $\R^2$. For $k\in \mf k$ a unit vector we have
\[
\varphi(k) = c_1\ad(h)|_\m + c_2 \ad(e_5)|_\m+ c_3 e_{67},
\]
where $e_6,e_7$ is an orthonormal basis of $\R^2$. The $(\mf k,B)$-extension is reducible precisely when $c_3=0$ or $c_1=c_2=0$. 
If $c_1\neq 0$, then the $(\mf k,B)$-extension defines a naturally reductive structure on
\[
(SU(2)\times SU(2)\times H^3)/\R_\alpha,
\]
where the image of $\Lie(\R_\alpha)$ in $\su(2)\oplus \su(2)$ is spanned by $h$ and in $\Lie(H^3)$ by the center, see \cite[Sec.~2.3]{Storm2018} for more details. 
Using an automorphism of $\g$ we can arrange that $c_1,c_2,c_3\geq 0$.
Under these extra assumptions all the naturally reductive structures are non-isomorphic. This $(\mf k,B)$-extension is regular for all values of $\alpha$ even though the base space is only regular when $\alpha\in \Q$. For every $\alpha\in \R \backslash\{0\}$ we obtain in this way a 5-parameter family of naturally reductive structures with $\lambda_1,\lambda_2>0$ and $c_1,c_2,c_3$ as parameters. 

If $c_1 = 0 $, then the naturally reductive structure is only regular when $\alpha = q\in \Q$. In this case the $(\mf k,B)$-extension defines a naturally reductive structure on
\[
(SU(2)\times SU(2))/S^1_q\times H^3.
\]
On this homogeneous space we obtain a 4-parameter family of naturally reductive structures, with $\lambda_1,\lambda_2>0$ and $c_2,c_3$ as parameters.

For both spaces we can replace one $S^2$ factor by its symmetric dual $SL(2,\R)/S^1$.

\vspace{3pt}
\paragraph{\textbf{The canonical base space is $SU(3)/(S^1\times S^1)$}}
We pick the following orthonormal basis with respect to $\overline{g}= \frac{-1}{12\lambda^2}B_{\su(3)}$ of $\h:=\Lie(S^1\times S^1)$:
\[
h_1:=\begin{pmatrix}
i\lambda & 0 & 0\\
0 & -i\lambda & 0\\
0 & 0 & 0
\end{pmatrix}\quad\mbox{and}\quad h_2:=\begin{pmatrix}
\frac{-i\lambda}{\sqrt{3}} & 0 & 0\\
0 & \frac{-i\lambda}{\sqrt{3}} & 0\\
0 & 0 & \frac{2i\lambda}{\sqrt{3}}
\end{pmatrix}.
\]
In this case we have $\varphi(\mf k) = \ad(\h)|_\m$. The only freedom is in the choice of a metric $B$ on $\mf k$. For $x_1,x_2,x_3\in \R$ we define a quadratic form on $\mc Z(\mf u(3))$ by
\[
\begin{pmatrix}
ia & 0 & 0\\
0 & ib & 0\\
0 & 0 & ic
\end{pmatrix}\mapsto \frac{x_1a^2+x_2b^2+x_3 c^2}{\lambda^2}.
\]
Restricting this to $\h$ gives us in the basis $h_1,h_2$ the following symmetric bilinear form:
\[
B_{x_1,x_2,x_3} := \begin{pmatrix}
x_1+x_2 & \frac{1}{\sqrt{3}}(-x_1+x_2)\\
\frac{1}{\sqrt{3}}(-x_1+x_2) & \frac{1}{3}(x_1+x_2+4x_3)
\end{pmatrix}.
\]
This is positive definite if and only if its trace and determinant are positive, i.e.
\[
\frac{3}{4}\textrm{tr}(B_{x_1,x_2,x_3})=x_1+x_2+x_3>0\quad\mbox{and}\quad \frac{3}{4}\det(B_{x_1,x_2,x_3})= x_1x_2+x_2x_3+x_1x_3>0.
\]
This parametrizes exactly all metric tensors on $\h$. From \Cref{prop:iso type II} we know that two of these metrics induce an isomorphic naturally reductive structure precisely when they are conjugate by an automorphism of $\su(3)$ which preserves $\h$, i.e. an element of the normalizer $N_{\su(3)}(\h)$ of $\h$ in $\su(3)$. Two metrics are conjugate by an element of $N_{\su(3)}(\h)$ if and only if they are conjugate by an element of the Weyl group of $\su(3)$. The Weyl group of $\su(3)$ is isomorphic to $S_3$ and the action of the Weyl group on $\h$ is given by permuting the diagonal entries. Therefore, the induced Weyl group action on the metrics $B_{x_1,x_2,x_3}$ simply permutes the indices. We see that under the conditions $x_3\geq x_2\geq x_1$ every $S_3$-orbit of these metrics is parametrized exactly ones. We still need to know when condition $(ii)$ from \Cref{lem:canonical base iff} is satisfied. The curvature of the $(\mf k,B)$-extension is given by
\[
R = -\ad(h_1)|_\m\odot \ad(h_1)|_\m -\ad(h_2)|_\m\odot \ad(h_2)|_\m +\sum_{i,j=1}^2 (B^{-1})_{ij} \ad(h_i)|_\m\odot \ad(h_j)|_\m.
\]
In the basis $h_1,h_2$ this becomes
\[
R=6\lambda^2\begin{pmatrix}
-1 & 0\\
0 & -1
\end{pmatrix} 
+ 
6\lambda^2 \det(B)^{-1}\begin{pmatrix}
\frac{1}{3}(x_1+x_2+4x_3) & \frac{1}{\sqrt{3}}(x_1-x_2)\\
\frac{1}{\sqrt{3}}(x_1-x_2) & x_1+x_2
\end{pmatrix}.
\]
This has full rank if and only if $x_1x_2+x_2x_3+x_1x_3-x_1-x_2-x_3+\frac{3}{4}\neq 0$.
Under this condition the canonical base space is equal to $SU(3)/(S^1\times S^1)$ by \Cref{lem:canonical base iff}.  The $(\mf k,B)$-extension is always regular and irreducible. Under the above conditions we obtain a 4-parameter family of naturally reductive structures on $SU(3)$, with $\lambda>0$ and $x_1,x_2,x_3$ as parameters. None of these structures are isomorphic under the condition $x_3\geq x_2\geq x_1$.

\vspace{3pt}
\paragraph{\textbf{The canonical base space is $\{*\}$}}
We write $\g=\{0\}$ for the 0-dimensional Lie algebra. Let $\mf k=\su(3)$ and let $B= \frac{-1}{\lambda^2} B_{\su(3)}$. Let $x_1\.x_8$ be an orthonormal basis of $\su(3)$ with respect to $B$. The torsion and curvature are given by
\[
T(x,y,z)=2B([x,y],z)\quad \mbox{and}\quad R =\sum_{i=1}^8 \ad(x_i)\odot \ad(x_i),
\]
The infinitesimal model is always irreducible and regular and defines a 1-parameter family of naturally reductive structures on $\R^8$ with $\lambda>0$ as parameter, see \cite{Storm2018}.

\vspace{1em}

We summarize the classification of all 7- and 8-dimensional naturally reductive spaces of type II in the following theorem.

\begin{theorem}\label{thm:classification type II}
	Every 7- and 8-dimensional simply connected naturally reductive space of type II for which the semisimple factor of the canonical base space is compact is presented in \Cref{tab:type II}. In the fourth column is the number of parameters of naturally reductive structures of type II on the homogeneous space $G/H$. The canonical base space of the naturally reductive structure is in the second column. The third column indicates if partial dual spaces exist.
\end{theorem}

\begin{notation}
	In \Cref{tab:type II} $H^n$ denotes the $n$-dimensional Heisenberg group and $QH^7$ denotes the $7$-dimensional quaternionic Heisenberg group. The subscripts $q_i\in \Q$ and $\alpha\in \R$ denote parameters which determine the subgroup, see \Cref{sec:Classification type I} for the details. Lastly for $\varphi:\mf k\to \so(n)$ a Lie algebra representation $Nil(\varphi)$ denotes a naturally reductive structure on the 2-step nilpotent Lie group as described in \cite{Gordon1985} and \cite[Sec.~2.2]{Storm2018}. 
\end{notation}

%
\begin{table}
	\[
	\begin{array}{| c | c | c | c | }
	\hline
	G/H & \mbox{canonical base space} & \mbox{dual} &  \#~ \mbox{param.} \\
	\hline
	\hline
	H^7 & \R^6 & \text{\xmark} & 3 \\
	\hline
	QH^7 & \R^4 & \text{\xmark} & 1 \\
	\hline
	S^2\times H^5 & S^2\times \R^4 & \text{\cmark} & 4 \\
	\hline
	S^2\times S^2 \times H^3 & S^2\times S^2\times \R^2 & \text{\cmark} & 5 \\
	\hline
	\C P^2\times H^3 & \C P^2 \times \R^2 & \text{\cmark} & 3 \\
	\hline
	Sp(2)/Sp(1)_{\mbox{st}} & Sp(2)/(SU(2)\times S^1) & \text{\xmark} & 2 \\
	\hline
	SU(3)/S^1_{q} & SU(3)/(S^1\times S^1) & \text{\xmark}  & 2\\
	\hline
	SO(5)/SO(3)_{\mbox{st}} & SO(5)/(SO(3)\times SO(2)) & \text{\cmark} & 2  \\
	\hline
	SU(4)/SU(3) & SU(4)/S(U(1)\times U(3)) & \text{\cmark} & 2\\
	\hline
	(SU(3)\times SU(2))/(SU(2)\times S^1_{q}) & \C P^2\times S^2 & \text{\cmark} & 3 \\
	\hline
	SU(2)^3/(S^1_{q_1}\times S^1_{q_2}) & S^2\times S^2\times S^2 & \text{\cmark} & 4  \\
	\hline
	\hline
	Nil(\R^2\to \so(6)) & \R^6 & \text{\xmark} & 5 \\
	\hline
	Nil(\varphi_{\mbox{ir}} : \so(3)\to \so(5)) & \R^5 & \text{\xmark} & 1 \\
	\hline
	Nil(\mf u(2) \to \so(4)) & \R^4 & \text{\xmark} & 2 \\
	\hline
	(SU(2)\times Nil(\R^2\to \so(4)))/\R & S^2\times \R^4 & \text{\cmark} & 6 \\
	\hline
	SU(2) \times H^5 & S^2\times \R^4 & \text{\cmark} & 5 \\
	\hline
	SU(2)\times H^5 & SU(2)\times \R^4 & \text{\xmark} & 4\\
	\hline
	SU(3)/SU(2)_{\mbox{st}} \times H^3 & \C P^2 \times \R^2 & \text{\cmark} & 4\\
	\hline
	(SU(2)\times SU(2)\times H^3)/\R_\alpha & S^2\times S^2 \times \R^2 & \text{\cmark} & 6 \\
	\hline
	(SU(2)\times SU(2))/S^1_{q} \times H^3 & S^2\times S^2 \times \R^2 & \text{\cmark} & 5 \\
	\hline
	(SU(2)\times SU(2)\times H^3)/\R_\alpha & (SU(2)\times SU(2))/S^1_{\alpha}\times \R^2 & \text{\cmark} & 5 \\
	\hline
	(SU(2)\times SU(2))/S^1_q\times H^3 & (SU(2)\times SU(2))/S^1_{q}\times \R^2 & \text{\cmark} & 4 \\
	\hline
	SU(3)/SU(2)_{\mbox{st}}\times H^3 & SU(3)/SU(2)_{\mbox{st}}\times \R^2 & \text{\xmark} & 3\\
	\hline
	SU(2)\times S^2 \times H^3 & SU(2)\times S^2\times \R^2 & \text{\cmark} & 5 \\
	\hline
	SU(3) & SU(3)/S^1_{q} & \text{\xmark} & 3 \\
	\hline
	SU(2)^3/S^1_{q_3} & SU(2)^3/(S^1_{q_1}\times S^1_{q_2}) & \text{\cmark} & 4 \\
	\hline
	SU(2)^3/S^1_{q_1,q_2} & (SU(2)\times SU(2))/S^1_{q}\times S^2 & \text{\cmark} & 4 \\
	\hline
	(SU(3)\times SU(2))/SU(2)_{\mbox{st}\times \mbox{id}} & (SU(3)\times SU(2))/(SU(2)_{\Delta}\times S^1) & \text{\cmark} & 3 \\
	\hline
	(SU(3)\times SU(2))/SU(2)_{\mbox{st}} & (SU(3)\times SU(2))/(SU(2)_{\mbox{st}}\times S^1_{q}) & \text{\cmark} & 3 \\
	\hline
	SU(3) & SU(3)/(S^1\times S^1) & \text{\xmark} & 4 \\
	\hline
	SU(3)/SU(2)_{\mbox{st}} \times SU(2) & \C P^2\times S^2 & \text{\cmark} & 5 \\
	\hline 
	SU(2)^3/S^1_{q_1,q_2} & S^2\times S^2\times S^2 & \text{\cmark} & 6 \\
	\hline
	\R^8 & \{*\} & \text{\xmark} & 1 \\
	\hline
	\end{array}
	\]
	\captionof{table}{7- and 8-dimensional type II spaces. \label{tab:type II}}
\end{table}

\proof[Acknowledgements]
This paper is part of my PhD thesis supervised by Professor Ilka Agricola, whom I would like to thank for her ongoing support and guidance.
\newpage

\bibliographystyle{abbrv_url}
\bibliography{./NaturalReductiveDim7}

\def\cprime{$'$}
\begin{thebibliography}{10}

\bibitem{AgricolaFerreiraFriedrich2015}
I.~Agricola, A.~C. Ferreira, and T.~Friedrich.
\newblock The classification of naturally reductive homogeneous spaces in
  dimensions {$n\leq 6$}.
\newblock {\em Differential Geom. Appl.}, 39:59--92, 2015.
\newblock URL: \url{http://dx.doi.org/10.1016/j.difgeo.2014.11.005}.

\bibitem{AgricolaFerreiraStorm15}
I.~Agricola, A.~C. Ferreira, and R.~Storm.
\newblock Quaternionic {H}eisenberg groups as naturally reductive homogeneous
  spaces.
\newblock {\em Int. J. Geom. Methods Mod. Phys.}, 12(8):1560007, 10, 2015.
\newblock URL: \url{http://dx.doi.org/10.1142/S0219887815600075}.

\bibitem{AmbroseSinger1958}
W.~Ambrose and I.~M. Singer.
\newblock On homogeneous {R}iemannian manifolds.
\newblock {\em Duke Math. J.}, 25:647--669, 1958.

\bibitem{CartanE1926}
E.~Cartan.
\newblock Sur une classe remarquable d'espaces de {R}iemann.
\newblock {\em Bull. Soc. Math. France}, 54:214--264, 1926.
\newblock URL: \url{http://www.numdam.org/item?id=BSMF_1926__54__214_0}.

\bibitem{D'AtriZiller1979}
J.~E. D'Atri and W.~Ziller.
\newblock Naturally reductive metrics and {E}instein metrics on compact {L}ie
  groups.
\newblock {\em Mem. Amer. Math. Soc.}, 18(215):iii+72, 1979.

\bibitem{FriedKathMorSem1997}
T.~Friedrich, I.~Kath, A.~Moroianu, and U.~Semmelmann.
\newblock On nearly parallel {$G\sb 2$}-structures.
\newblock {\em J. Geom. Phys.}, 23(3-4):259--286, 1997.
\newblock URL: \url{http://dx.doi.org/10.1016/S0393-0440(97)80004-6}.

\bibitem{Gordon1985}
C.~S. Gordon.
\newblock Naturally reductive homogeneous {R}iemannian manifolds.
\newblock {\em Canad. J. Math.}, 37(3):467--487, 1985.
\newblock URL: \url{http://dx.doi.org/10.4153/CJM-1985-028-2}.

\bibitem{Kaplan1983}
A.~Kaplan.
\newblock On the geometry of groups of {H}eisenberg type.
\newblock {\em Bull. London Math. Soc.}, 15(1):35--42, 1983.
\newblock URL: \url{https://doi.org/10.1112/blms/15.1.35}.

\bibitem{Kostant1956}
B.~Kostant.
\newblock On differential geometry and homogeneous spaces. {I}, {II}.
\newblock {\em Proc. Natl. Acad. Sci. U.S.A.}, 42:258--261, 354--357, 1956.

\bibitem{Kowalski1990}
O.~Kowalski.
\newblock Counterexample to the ``second {S}inger's theorem''.
\newblock {\em Ann. Global Anal. Geom.}, 8(2):211--214, 1990.
\newblock URL: \url{http://dx.doi.org/10.1007/BF00128004}.

\bibitem{KowalskiVanhecke1983}
O.~Kowalski and L.~Vanhecke.
\newblock Four-dimensional naturally reductive homogeneous spaces.
\newblock {\em Rend. Sem. Mat. Univ. Politec. Torino}, (Special Issue):223--232
  (1984), 1983.
\newblock Conference on differential geometry on homogeneous spaces (Turin,
  1983).

\bibitem{KowalskiVanhecke1985}
O.~Kowalski and L.~Vanhecke.
\newblock Classification of five-dimensional naturally reductive spaces.
\newblock {\em Math. Proc. Cambridge Philos. Soc.}, 97(3):445--463, 1985.
\newblock URL: \url{http://dx.doi.org/10.1017/S0305004100063027}.

\bibitem{Nomizu1954}
K.~Nomizu.
\newblock Invariant affine connections on homogeneous spaces.
\newblock {\em Amer. J. Math.}, 76:33--65, 1954.

\bibitem{Storm2018}
R.~Storm.
\newblock A new construction of naturally reductive spaces.
\newblock {\em Transform. Groups}, 23(2):527--553, Jun 2018.
\newblock URL: \url{https://doi.org/10.1007/s00031-017-9446-5}.

\bibitem{Storm2018a}
R.~Storm.
\newblock Structure theory of naturally reductive spaces.
\newblock 2018.
\newblock \href {http://arxiv.org/abs/1810.02616} {\path{arXiv:1810.02616}}.

\bibitem{Tricerri1992}
F.~Tricerri.
\newblock Locally homogeneous {R}iemannian manifolds.
\newblock {\em Rend. Sem. Mat. Univ. Politec. Torino}, 50(4):411--426 (1993),
  1992.
\newblock Differential geometry (Turin, 1992).

\bibitem{TricerriVanhecke1983}
F.~Tricerri and L.~Vanhecke.
\newblock {\em Homogeneous structures on {R}iemannian manifolds}, volume~83 of
  {\em London Mathematical Society Lecture Note Series}.
\newblock Cambridge University Press, Cambridge, 1983.
\newblock URL: \url{http://dx.doi.org/10.1017/CBO9781107325531}.

\bibitem{Wilking1999}
B.~Wilking.
\newblock The normal homogeneous space {$({\rm SU}(3)\times {\rm SO}(3))/{\rm
  U}^\bullet (2)$} has positive sectional curvature.
\newblock {\em Proc. Amer. Math. Soc.}, 127(4):1191--1194, 1999.
\newblock URL: \url{https://doi.org/10.1090/S0002-9939-99-04613-4}.

\end{thebibliography}

\end{document}